\documentclass[a4paper,notitlepage,twoside,leqno,10pt]{amsart}

\usepackage{bbm,pifont,latexsym}

\usepackage{dcolumn,indentfirst}
\usepackage[bookmarksopen=true,linktocpage=true,pdfstartview={XYZ null null 1.25}]{hyperref}
\usepackage{amsmath,amssymb,amscd,amsthm,amsfonts,mathrsfs}
\usepackage{color,graphicx,xcolor,graphics,subfigure,extarrows,caption2}
\usepackage{titletoc}

\newtheorem{thm}{Theorem}[section]
\newtheorem{lema}[thm]{Lemma}
\newtheorem{cor}[thm]{Corollary}
\newtheorem{prop}[thm]{Proposition}

\theoremstyle{definition}
\newtheorem*{defi}{Definition}

\newcommand{\D}{\mathbb{D}}
\newcommand{\T}{\mathbb{T}}
\newcommand{\R}{\mathbb{R}}

\newcommand{\C}{\mathbb{C}}
\newcommand{\N}{\mathbb{N}}
\newcommand{\EC}{\widehat{\mathbb{C}}}

\newcommand{\diam}{\textup{diam}}
\newcommand{\dist}{\textup{dist}}

\newcommand{\Int}{\textup{int}}

\newcommand{\ii}{\textup{i}}

\usepackage{enumerate,enumitem}

\makeatletter\@addtoreset{equation}{section}\makeatother

\begin{document}

\author[W. Qiu]{WEIYUAN QIU}
\address{School of Mathematical Sciences, Fudan University, Shanghai 200433, P. R. China}
\email{wyqiu@fudan.edu.cn}

\author[F. Yang]{FEI YANG}
\address{Department of Mathematics, Nanjing University, Nanjing 210093, P. R. China}
\email{yangfei@nju.edu.cn}

\author[Y. Yin]{YONGCHENG YIN}
\address{School of Mathematical Sciences, Zhejiang University, Hangzhou 310027, P. R. China}
\email{yin@zju.edu.cn}

\title[QUASISYMMETRIC GEOMETRY OF JULIA SETS]{QUASISYMMETRIC GEOMETRY OF THE JULIA SETS OF MCMULLEN MAPS}

\begin{abstract}
We study the quasisymmetric geometry of the Julia sets of McMullen maps $f_\lambda(z)=z^m+\lambda/z^\ell$, where $\ell$, $m\geq 2$ are integers satisfying $1/\ell+1/m<1$ and $\lambda\in\mathbb{C}\setminus\{0\}$. If the free critical points of $f_\lambda$ are escaped to the infinity, we prove that the Julia set $J_\lambda$ of $f_\lambda$ is quasisymmetrically equivalent to either a standard Cantor set, a standard Cantor set of circles or a round Sierpi\'{n}ski carpet (which is also standard in some sense). If the free critical points are not escaped, we give a sufficient condition on $\lambda$ such that $J_\lambda$ is a Sierpi\'{n}ski carpet and prove that most of them are quasisymmetrically equivalent to some round carpets. In particular, there exist infinitely renormalizable rational maps whose Julia sets are quasisymmetrically equivalent to round carpets.
\end{abstract}



\date{\today}


\dedicatory{Dedicated to the memory of Professor Tan Lei}

\maketitle

\tableofcontents

\section{Introduction}

Let $\ell,m\geq 2$ be two positive integers satisfying $1/\ell+1/m<1$ and $I_0=[0,1]\subset\mathbb{R}$ the closed unit interval in the real line. We subdivide $I_0$ into 3 subintervals with sizes $1/\ell$, $1-1/\ell-1/m$, $1/m$ from left to right in the obvious way and then remove the interior of middle subinterval. The resulting set $I_1$ is the union of two intervals $[0,1/\ell]\cup[1-1/m,1]$. Let $g_0(x)=(1-x)/\ell$ and $g_1(x)=1+(x-1)/m$, where $x\in[0,1]$. We have $I_1=g_0(I_0)\cup g_1(I_0)$. For each $n\geq 1$, define $I_n:=g_0(I_{n-1})\cup g_1(I_{n-1})$. Then each $I_n$ is compact and inductively $I_n\subset I_{n-1}$. The \textit{standard} Cantor set $C_{\ell,m}$ is the intersection of all $I_n$, which is compact and totally disconnected. In particular, $C_{3,3}$ is the \emph{standard} middle third Cantor set.

The \textit{standard Cantor circles} $A_{\ell,m}$ is defined by $C_{\ell,m}\times\mathbb{T}$, where $\mathbb{T}:=\{z\in\mathbb{C}:|z|=1\}$ is the unit circle. A compact subset on the Riemann sphere $\EC =\mathbb{C}\cup\{0\}$ is called a \textit{Cantor set of circles} (or \textit{Cantor circles} in short) if there exists a homeomorphism between this compact set and the standard middle third Cantor circles $A_{3,3}$.

According to \cite{Why58}, a connected and locally connected compact set $S$ in $\EC$ is called a \emph{Sierpi\'{n}ski carpet} (or \emph{carpet} in short) if it has empty interior and can be written as $S=\EC \setminus \bigcup_{i\in\mathbb{N}}D_i$, where $\{D_i\}_{i\in\mathbb{N}}$ are Jordan regions satisfying $\partial D_i\cap\partial D_j=\emptyset$ for $i\neq j$ and the spherical diameter $\text{diam}(\partial D_i)\rightarrow 0$ as $i\rightarrow \infty$. The collection of the boundaries of the Jordan regions $\{\partial D_i\}_{i\in\N}$ are called the \emph{peripheral circles} of $S$. All Sierpi\'{n}ski carpets are homeomorphic to each other. A carpet is called \textit{round} if the peripheral circles are all spherical (or Euclidean) circles.

If it is known that two metric spaces $(X,d_X)$ and $(Y,d_Y)$ are homeomorphic, a basic question in quasiconformal geometry is to determine whether these two spaces are \textit{quasisymmetrically equivalent} (\textit{q.s.\,equivalent} in short) to each other. This means that there exist a homeomorphism $f:X\rightarrow Y$ and a distortion control function $\eta:[0,\infty)\rightarrow [0,\infty)$ which is also a homeomorphism such that
\begin{equation*}
\frac{d_Y(f(x),f(a))}{d_Y(f(x),f(b))}\leq \eta\left(\frac{d_X(x,a)}{d_X(x,b)}\right)
\end{equation*}
for every distinct points $x$, $a$, $b\in X$ (see for example, \cite[p.\,110]{Hei01}).

The study of the topological properties of the Julia sets of rational maps, including the connectivity and locally connectivity etc., is one of the important problems in complex dynamics. However, for the study of the quasisymmetric properties of the Julia sets of rational maps, the corresponding results appear much fewer. The first example of quasisymmetrically inequivalent hyperbolic Julia sets was given by Ha\"{i}ssinsky and Pilgrim in \cite{HP12b}. They proved that there exist quasisymmetrically inequivalent Cantor circles as the Julia sets of rational maps. Recently, Bonk, Lyubich and Merenkov studied the quasisymmetries of the Sierpi\'{n}ski carpets which arise as the Julia sets of critically finite rational maps \cite{BLM16}. See also \cite{QYZ19} for the generalization of the corresponding results to the critically infinite case. In this article, we study the quasisconformal geometry and give a quasisymmetric classification of the Julia sets of \emph{McMullen maps}
\begin{equation}\label{McMullen}
f_\lambda(z)=z^m+\lambda/z^\ell,
\end{equation}
where $\ell$, $m\geq 2$ are integers satisfying $1/\ell+1/m<1$ and $\lambda\in\mathbb{C}^*:=\mathbb{C}\setminus\{0\}$. The dynamical behaviors of $f_\lambda$ have been studied extensively recently (see \cite{Dev13, DL05, DLU05, DR13, QWY12, QXY12} and the references therein).

For the topological properties of the Julia set $J_\lambda$ of $f_\lambda$, Devaney, Look and Uminsky proved an Escape Trichotomy Theorem which asserts that if all the critical points of $f_\lambda$ are attracted by $\infty$, then $J_\lambda$ is either a Cantor set, a Cantor set of circles or a Sierpi\'{n}ski carpet (see \cite[Theorem 0.1]{DLU05} or Theorem \ref{E-T-T}). For the quasisymmetric geometric properties of $J_\lambda$, we have the following theorem.

\begin{thm}\label{qs-classifi}
Suppose that all the critical points of $f_\lambda$ are attracted by $\infty$. Then one and only one of the following three cases happens:
\begin{enumerate}
\item $J_\lambda$ is a Cantor set and q.s.\,equivalent to the standard middle third Cantor set $C_{3,3}$;
\item $J_\lambda$ is a Cantor set of circles and q.s.\,equivalent to the standard Cantor circles $A_{\ell,m}$;
\item $J_\lambda$ is a Sierpi\'{n}ski carpet and q.s.\,equivalent to a round carpet.
\end{enumerate}
\end{thm}

Except 0 and $\infty$, the remaining $\ell+m$ critical points of $f_\lambda$ are called the \emph{free} critical points. These free critical points have the same orbit essentially (see \S \ref{escape-case}). They are attracted to $\infty$ or not at the same time. If the free critical points are not attracted to $\infty$, then $J_\lambda$ is connected (see \cite{DR13}). In this case, two natural questions arise: When $J_\lambda$ is a Sierpi\'{n}ski carpet? Is it q.s. equivalent to a round one?

The study of the quasisymmetric equivalences between the Sierpi\'{n}ski carpets and round carpets was partially motivated by the Kapovich-Kleiner conjecture in the geometry group theory. This conjecture is equivalent to the following statement: if the boundary of the infinity $\partial_\infty G$ of a Gromov hyperbolic group $G$ is a Sierpi\'{n}ski carpet, then $\partial_\infty G$ is q.s. equivalent to a round carpet in $\EC$. Recently, Bonk gave a sufficient condition on the Sierpi\'{n}ski carpets such that they can be q.s. equivalent to some round carpets (see \cite{Bon11}).

If the free critical orbits of $f_\lambda$ are bounded and one of them is attracted by an attracting periodic orbit, then $J_\lambda$ is a Sierpi\'{n}ski carpet if its Fatou components satisfy some `buried' properties (see \cite{DL05}). In order to find more specific $\lambda$ such that $J_\lambda$ is a Sierpi\'{n}ski carpet, we focus our attention on the global parameter space of $f_\lambda$. Define the \textit{non-escaping locus} of this family by
\begin{equation}\label{Lambda}
\Lambda=\{\lambda\in\mathbb{C}^*:~\text{The free critical orbits of $f_\lambda$ are bounded}\}.
\end{equation}
It was known that $\Lambda$ is connected and contains infinitely many small homeomorphic copies of the Mandelbrot set which correspond to the \textit{renormalizable} parameters (see \cite{Dev06,Ste06} and Figure \ref{Fig_McMullen_para}).

\begin{figure}[!htpb]
  \setlength{\unitlength}{1mm}
  \centering
  \includegraphics[height=55mm]{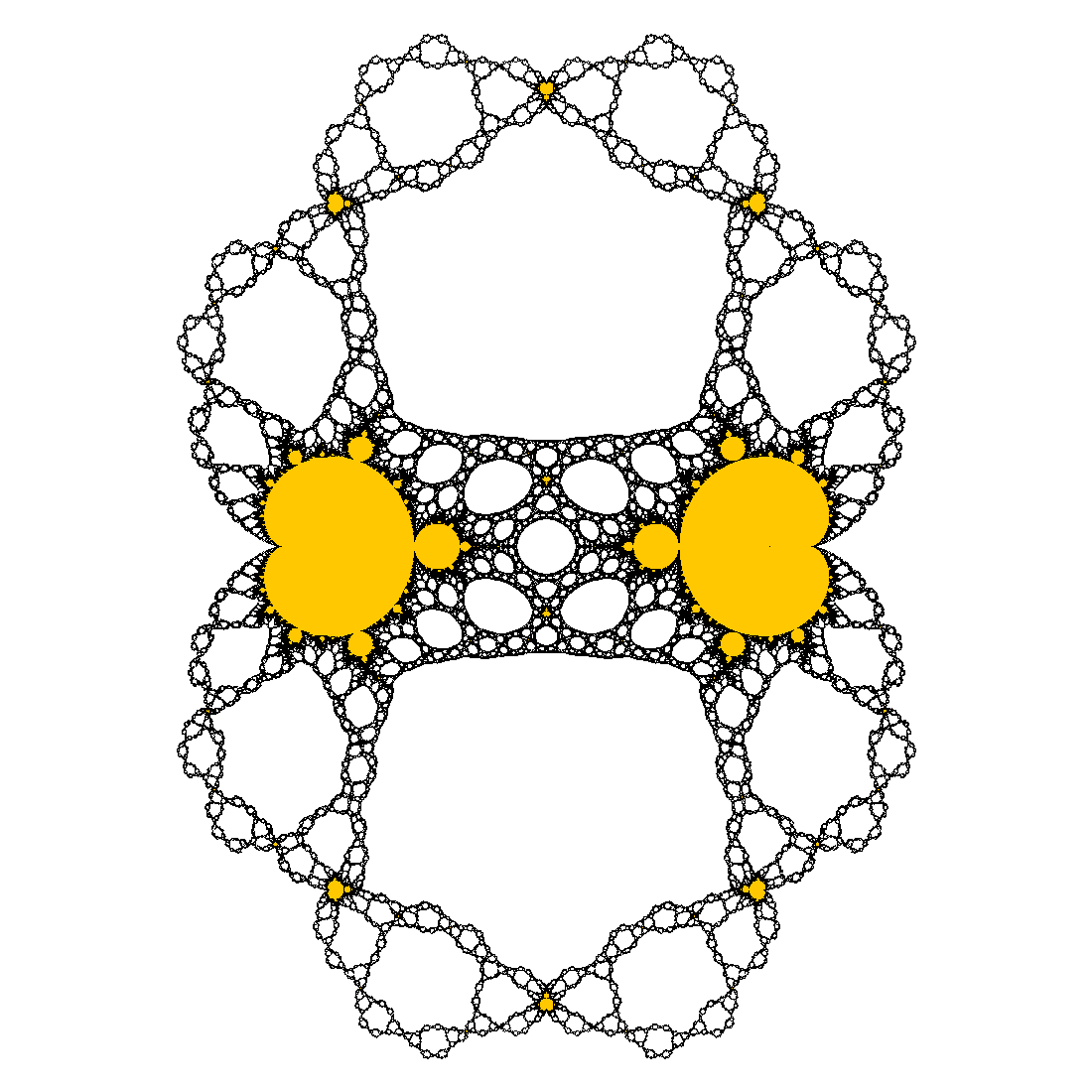}\hskip0.2cm
  \includegraphics[height=55mm]{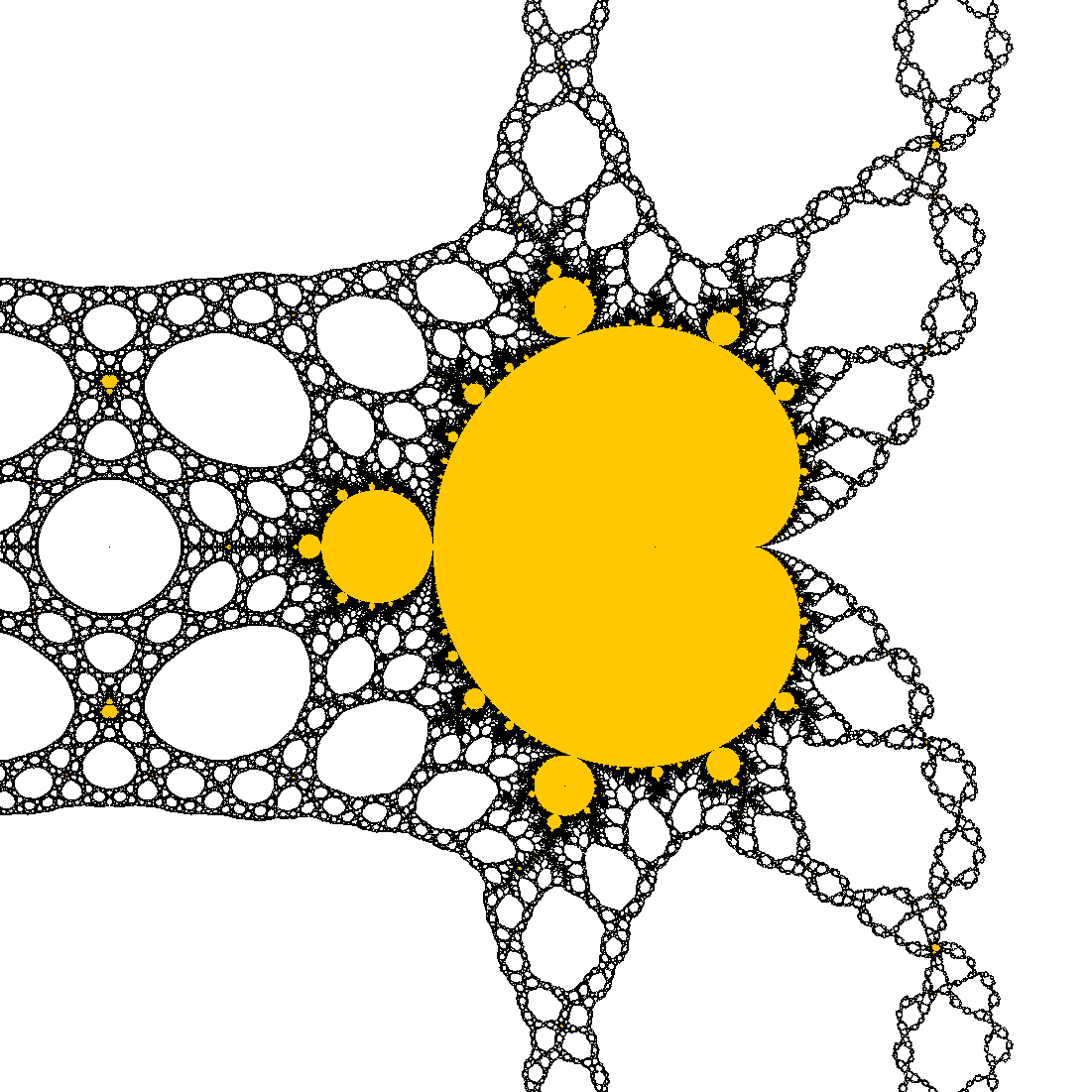}
  \caption{The non-escaping locus of $f_\lambda(z)=z^3+\lambda/z^3$ and its zoom near the positive real axis. A baby Mandelbrot set can be seen clearly in the picture on the right.}
  \label{Fig_McMullen_para}
\end{figure}

The Mandelbrot set $M$ is the collection of all the parameters $c$ such that the Julia set of $P_c(z)=z^2+c$ is connected. Let $H$ be a \emph{hyperbolic component} of $M$. For each $c\in H$, $P_c$ has an attracting periodic orbit with period depending only on $H$ but not on $c$. We use $p(H)$ to denote this period. There exists a unique parameter $r_H$ on the boundary of $H$ such that $P_{r_H}$ has a parabolic Fatou component with period $p(H)$. This point $r_H$ is called the \emph{root} of $H$. In particular, let
\begin{equation*}
H_\heartsuit=\{c\in M: P_c \text{ has an attracting fixed point}\}
\end{equation*}
be the hyperbolic component of $M$ with period 1. The subscript $\heartsuit$ depicts the shape of $H_\heartsuit$, which is surrounded by a cardioid.
The hyperbolic components of $M$ can be divided into two types: \textit{satellite} and \textit{primitive} (see definitions in \S \ref{Non-escap-case}). An intuitive difference between these two types of components is: the primitive hyperbolic component is surrounded by a cardioid, but the satellite one is not.

Since each hyperbolic component $\mathcal{H}$ in $\Lambda$ is contained in some maximal homeomorphic image of the Mandelbrot set (see \cite[Theorem 10]{Ste06}), we can also define the corresponding \textit{root} $r_{\mathcal{H}}$ of $\mathcal{H}$ and divide the hyperbolic components in $\Lambda$ into two types: satellite and primitive (see \S \ref{Non-escap-case}).

The question when $J_\lambda$ is a Sierpi\'{n}ski carpet was considered in the case of $\lambda>0$ in \cite{QXY12}. Let $\R^+$ be the positive real axis. Note that the non-escaping locus $\Lambda$ is compact. We denote
\begin{equation}\label{equ-lambda-0-1}
\lambda_0:=\min \{\Lambda\cap\R^+\} \text{ and } \lambda_1:=\max \{\Lambda\cap\R^+\}.
\end{equation}
Let $K_\lambda=\EC\setminus \mathcal{A}(\infty)$ be the \emph{filled-in Julia set} of $f_\lambda$, where $\mathcal{A}(\infty)$ is the attracting basin of $\infty$.

\begin{thm}[\cite{QXY12, Xie11}]\label{Qiu-Xie-Yin}
If $\lambda\in[\lambda_0,\lambda_1]$, then $J_\lambda$ is locally connected. Further, $J_\lambda$ is a Sierpi\'{n}ski carpet if and only if one of the following holds:
\begin{enumerate}
\item The interior $\Int (K_\lambda)=\emptyset$ and $\lambda\neq\lambda_0$; or
\item $\lambda\in\mathcal{H} \cup \{r_{\mathcal{H}}\}$, where $\mathcal{H}$ is a primitive hyperbolic component of a maximal copy of $M$ in $\Lambda$ such that $\mathcal{H}$ is not the image of $H_\heartsuit$ under the maximal homeomorphic map and $r_{\mathcal{H}}$ is the root of $\mathcal{H}$.
\end{enumerate}
\end{thm}

The precise definition of the maximal homeomorphic map defined from the Mandelbrot set to $\Lambda$ will be given in $\S \ref{Non-escap-case}$. Theorem \ref{Qiu-Xie-Yin} was proved only in the case $\ell=m\geq 3$ in \cite{QXY12}. Then Xie extended this result to the general case in her Ph.D. thesis and the idea of the proof is almost the same (see \cite{Xie11}). In this article, we will prove the following theorem.

\begin{thm}\label{QS-positive}
If $\lambda\in[\lambda_0,\lambda_1]$, then $J_\lambda$ is q.s.\,equivalent to a round carpet if and only if $\lambda$ satisfies \textup{(a)} or \textup{(b)} but $\lambda\neq r_{\mathcal{H}}$ in Theorem \ref{Qiu-Xie-Yin}.
\end{thm}

As a corollary, we have

\begin{cor}\label{non-hype-carpet}
If $\lambda\in(\lambda_0,\lambda_1)$ such that one of the free critical points of $f_\lambda$ is strictly pre-periodic, then $J_\lambda$ is q.s.\,equivalent to a round carpet. Moreover, there exists infinitely renormalizable rational map whose Julia set is q.s.\,equivalent to a round carpet.
\end{cor}

We will consider more general cases such that $J_\lambda$ is a Sierpi\'{n}ski carpet. In particular, the parameter $\lambda$ can vary in the punctured complex plane but not only in the positive real line. However, this requires us to set $\ell=m\geq 3$ since the local connectivity of the Julia set $J_\lambda$ was only known for $\ell= m$ if $\lambda$ is complex.

\begin{thm}\label{QS-complex}
Suppose that $\ell=m\geq 3$ and the free critical orbits of $f_\lambda$ are bounded. Then $J_\lambda$ is a Sierpi\'{n}ski carpet if
\begin{enumerate}
\item The interior $\Int (K_\lambda)=\emptyset$ and the closure of the free critical orbits does not intersect with the boundary of the immediate attracting basin of $\infty$; or
\item $\lambda\in\mathcal{H} \cup \{r_{\mathcal{H}}\}$, where $\mathcal{H}$ is a primitive hyperbolic component of a maximal copy of $M$ in $\Lambda$ such that $\mathcal{H}$ is not the image of $H_\heartsuit$ under the maximal homeomorphic map.
\end{enumerate}

Moreover, $J_\lambda$ is q.s.\,equivalent to a round carpet if $\lambda$ satisfies \textup{(a)} or \textup{(b)} but $\lambda\neq r_{\mathcal{H}}$.
\end{thm}

This article is organized as follows: In \S\ref{escape-case}, we state some preliminary properties of the McMullen maps and prove Theorem \ref{qs-classifi}. In \S\ref{sec-qs-in-cc}, as an extension of Ha\"{\i}ssinsky and Pilgrim's result, we give more quasisymmetrically inequivalent Cantor circle Julia sets of rational maps. In \S \ref{Non-escap-case}, we give the definitions of renormalization, the classification of the hyperbolic components, the definition of maximal homeomorphic map and then prove Theorem \ref{QS-positive}, Corollary \ref{non-hype-carpet} and Theorem \ref{QS-complex}.

\medskip
\noindent\textit{Acknowledgements.} This work was supported by the NSFC (grant Nos.\,11671091, 11731003, 11401298, 11671092 and 11771387) and NSF of Jiangsu Province (grant No.\,BK20140587). We would like to thank the referees for their careful reading and helpful comments. We are also very grateful to Xiaoguang Wang for pointing out an error in the proof of Theorem \ref{unif-quasicircle-sep} in the earlier version.

\section{Preliminaries and the escaping case}\label{escape-case}

The map $f_\lambda$ has a critical point at the origin with multiplicity $\ell-1$, a critical point at $\infty$ with multiplicity $m-1$, and a simple critical point at
\begin{equation*}
\omega_j=e^{2\pi\ii\frac{ j}{\ell+m}}({\lambda \ell}/{m})^{\frac{1}{\ell+m}}, \text{~where~} 0\leq j< \ell+m.
\end{equation*}
These critical points $\omega_0,\cdots,\omega_{\ell+m-1}$ are called the \textit{free} critical points of $f_\lambda$. By a direct calculation, it follows that $|f_\lambda^{\circ n}(\omega_j)|=|f_\lambda^{\circ n}(\omega_k)|$ for $0\leq j,k<\ell+m$ and $n\geq 0$. Hence the fate of the Julia set of $f_\lambda$ is determined by any one of the $\ell+m$ free critical orbits.

\subsection{Escape trichotomy theorem}

Note that $f_\lambda$ has a super-attracting fixed point at $\infty$ whose preimages are $0$ and itself.
Let $B_\lambda$ and $T_\lambda$ be the Fatou components of $f_\lambda$ containing $\infty$ and $0$, respectively. Then $B_\lambda\cap T_\lambda=\emptyset$ or $B_\lambda=T_\lambda$. Devaney, Look and Uminsky gave a topological classification on the Julia sets as in the following theorem.

\begin{thm}[\cite{DLU05}]\label{E-T-T}
Suppose that the free critical points of $f_\lambda$ are attracted by $\infty$. Then one and only one of the following three cases happens:
\begin{enumerate}
\item $f_\lambda(\omega_j)\in B_\lambda$ for some $j$, then $J_\lambda$ is a Cantor set;
\item $f_\lambda(\omega_j)\in T_\lambda\neq B_\lambda$ for some $j$, then $J_\lambda$ is a Cantor set of circles;
\item $f_\lambda^{\circ k}(\omega_j)\in T_\lambda\neq B_\lambda$ for some $j$ and $k\geq 2$, then $J_\lambda$ is a Sierpi\'{n}ski carpet.
\end{enumerate}
\end{thm}

In fact, the parameter space $\mathbb{C}^*$ of $f_\lambda$ can be divided into four parts $\Lambda\sqcup\Lambda_\infty \sqcup\Lambda_0 \sqcup\Lambda_S$ (see \cite[\S 7]{Ste06} and `$\sqcup$' denotes the disjoint union), where $\Lambda$ is the \textit{non-escaping locus} defined in (\ref{Lambda}); $\Lambda_\infty=\{\lambda\in\mathbb{C}^*:J_\lambda \textup{ is a Cantor set}\}$ is a unbounded domain called the \textit{Cantor locus}; $\Lambda_0:=\{\lambda\in\mathbb{C}^*:J_\lambda \textup{ is a Cantor set of circles}\}$ is a punctured neighborhood of 0 called the \textit{McMullen domain} since McMullen is the first one who found the rational maps whose Julia sets are Cantor circles (see \cite[\S 7]{McM88}); and the components of $\Lambda_S=\{\lambda\in\mathbb{C}^*\setminus\Lambda:J_\lambda \textup{ is a Sierpi\'{n}ski carpet}\}$ are called \textit{Sierpi\'{n}ski holes}. Moreover, every hyperbolic component in $\Lambda$ is contained in some homeomorphic image of the Mandelbrot set (see \cite[Theorem 10]{Ste06} and Figure \ref{Fig_McMullen_para}).

\subsection{Q.S. uniformization of Cantor sets}

A proposition of David and Semmes \cite[Proposition 15.11]{DS97} asserts that if the Julia set of a hyperbolic rational map is homeomorphic to the middle third Cantor set $C_{3,3}$, then this Julia set is q.s. equivalent to $C_{3,3}$. By Theorem \ref{E-T-T}(a), this means that Theorem \ref{qs-classifi}(a) holds.

\subsection{Q.S. uniformization of Cantor circles}

We will prove Theorem \ref{qs-classifi}(b) by using quasiconformal surgery.
For any two positive integers $\ell,m\geq 2$ such that $1/\ell+1/m<1$, we define an \emph{iterated function system} (\textit{IFS} in short) on an annulus such that the \emph{attractor} of this IFS is conformally equivalent to the standard Cantor circles $A_{\ell,m}$. Then we extend the inverse of this IFS to a quasiregular map which is defined from $\EC $ to itself. Finally, after constructing an invariant Beltrami coefficient on $\EC $ and straightening the corresponding complex structure, we show that there exists a quasiconformal map conjugating this quasiregular map to the McMullen map $f_\lambda$ for some small $\lambda\neq 0$.

Fix $0<r_0<1$, define $I=[r_0,1]$. Let $I_0=[r_0,r_0^{1-1/\ell}]$ and $I_1=[r_0^{1/m},1]$ be two subintervals of $I$. Define three annuli $A:=\{re^{\ii t}:r\in I~\text{and}~0\leq t<2\pi\}$ and $A_j:=\{re^{\ii t}:r\in I_j~\text{and}~0\leq t<2\pi\}$ for $j=0,1$. Let $F:A_0\sqcup A_1\rightarrow A$ be the map such that
\begin{equation*}
F|_{A_0}(z)=r_0^\ell/z^l~~~\text{and}~~~
F|_{A_1}(z)=z^m.
\end{equation*}
Note that $F|_{A_0}:A_0\rightarrow A$ and $F|_{A_1}:A_1\rightarrow A$ are covering maps with degree $\ell$ and $m$ respectively. It is easy to see that the attractor of the IFS $\{(F|_{A_0})^{-1},(F|_{A_1})^{-1}\}$ is conformally equivalent to the standard Cantor circles $A_{\ell,m}$.

For each $r>0$, we use $\mathbb{D}_r:=\{z:|z|<r\}$ to denote the round disk centered at the origin with radius $r$. We extend the domain of definition of $F$ to the disk $\mathbb{D}_{r_0}$ by letting $F(z)=r_0^\ell/z^\ell$ which maps $\mathbb{D}_{r_0}$ onto the exterior of the closed unit disk $\EC \setminus\overline{\mathbb{D}}$ and define $F(z)=z^m$ which maps $\EC \setminus\overline{\mathbb{D}}$ onto itself. Obviously, the map $F$ is continuous in $\{z\in\EC:|z|\leq r_1\text{~or~}|z|\geq r_2\}$, where $r_1=r_0^{1-1/\ell}$ and $r_2=r_0^{1/m}$. Next, we want to extend $F$ such that $F$ maps the annulus $\mathbb{A}_{r_1,r_2}:=\{z:r_1<|z|<r_2\}$ onto the disk $\mathbb{D}_{r_0}$.

Recall that a branched covering $f:\EC \rightarrow\EC $ is called \textit{quasiregular} if it can be written as $\psi\circ R\circ\varphi$, where $R$ is a rational map and $\psi$, $\varphi$ are both quasiconformal homeomorphisms. For $r>0$, we use $\mathbb{T}_r=\{z:|z|=r\}$ to denote the boundary of $\D_r$. Let $A$ and $B$ be subsets of $\EC$. We use the notation `$A\Subset B$' if the closure $\overline{A}$ is contained in the interior of $B$.

\begin{lema}\label{an-to-disk}
Let $f_1(z)=r_0 r_1^\ell/z^\ell$ and $f_2(z)=r_0 z^m/r_2^m$ be two covering maps defined on the two circles $\mathbb{T}_{r_1}$ and $\mathbb{T}_{r_2}$ respectively, such that $f_1(\mathbb{T}_{r_1})=\mathbb{T}_{r_0}$ and $f_2(\mathbb{T}_{r_2})=\mathbb{T}_{r_0}$. Then there exists a continuous extension $F:\overline{\mathbb{A}}_{r_1,r_2}\rightarrow \overline{\mathbb{D}}_{r_0}$ such that
\begin{enumerate}
\item $F|_{\mathbb{T}_{r_i}}=f_i$, where $i=1,2$;
\item $F:\mathbb{A}_{r_1,r_2}\rightarrow \mathbb{D}_{r_0}$ is quasiregular; and
\item $F(e^{2\pi\ii/(\ell+m)}z)=e^{2m\pi\ii/(\ell+m)}F(z)$ for $z\in \overline{\mathbb{A}}_{r_1,r_2}$.
\end{enumerate}
\end{lema}

\begin{proof}
In order to state the process of construction more clearly, we only give the proof of the case for $\ell=2$ and $m=3$. The other cases can be proved completely similarly\footnote{The proof needs to be modified slightly if $\ell$ and $m$ are not coprime to each other.}. We first construct a branched covering from $\overline{\mathbb{A}}_{\varepsilon,1}$ to $\D$ with degree $\ell+m$, where $\varepsilon>0$ is small enough (i.e. we first assume that $r_1=\varepsilon$, $r_2=1$ and $r_0=1$ in the conditions of Lemma \ref{an-to-disk}). Then the desired extension in this lemma can be obtained by doing a radial setting.

Let $A$ be the domain bounded by a regular pentagram in $\mathbb{D}$ which is centered at the origin such that $\mathbb{D}_\varepsilon\Subset A$, where $x_0$, $x_1$, $\cdots$, $x_9$ are 10 vertices of this pentagram in anticlockwise order and $x_0$ is one of the outer five vertices with argument 0 (see the picture in Figure \ref{Fig_annu-to-disk} on the left). For $1\leq j\leq 5$, let $B_j$ be the regular pentagon containing two edges $[x_{2j-2},x_{2j-1}]$ and $[x_{2j-1},x_{2j}]$, where $x_{10}=x_0$. Since $\varepsilon$ is small enough, we can choose $A$ with appropriate size such that $A\cup\bigcup_{j=1}^5 B_j\Subset \mathbb{D}$. For $1\leq j\leq 5$, let $y_{2j-2},y_{2j-1}$ be other two vertices of $B_j$ in anticlockwise order which have not been marked. For $1\leq j\leq 5$ and $2j-2\leq k\leq 2j-1$, let $\gamma_{k+j}=[y_k,e^{2\pi (k+j)\ii/15}]$ be the segment connecting $y_k$ and $e^{2\pi (k+j)\ii/15}$. For $0\leq n< 5$, let $\gamma_{3n}=[x_{2n},e^{2\pi n\ii/5}]$. This means that we have 15 segments $\gamma_s$ which divide $E:=\mathbb{D}\setminus (A\cup\bigcup_{j=1}^5 B_j)$ into 15 topological quadrilaterals $E_s$ in anticlockwise order such that $\gamma_s$ and $\gamma_{s+1}$ are two edges of $E_s$, where $0\leq s<15$ and $\gamma_{15}:=\gamma_0$. Similarly, let $\eta_t=[x_t,\varepsilon e^{2\pi t\ii/10}]$ be the segment connecting $x_t$ and $e^{2\pi t\ii/10}$, where $0\leq t<10$. Then this 10 segments $\eta_t$ divide $D:=A \setminus\mathbb{D}_\varepsilon$ into 10 topological quadrilaterals $D_t$ in anticlockwise order such that $\eta_t$ and $\eta_{t+1}$ are two boundaries of $D_t$, where $0\leq t<10$ and $\eta_{10}:=\eta_0$.

Let $B$ be a regular pentagon in $\mathbb{D}$ which is centered at the origin, where $z_0$, $z_1$, $\cdots$, $z_4$ are 5 vertices of $B$ in anticlockwise order and the argument of $z_0$ is 0. Let $\zeta_n=[z_n,e^{2\pi n\ii/5}]$ be the segment connecting $z_n$ and $e^{2\pi n\ii/5}$, where $0\leq n<5$. Then this 5 segments $\zeta_n$ divide $\mathbb{D}\setminus B$ into 5 topological quadrilaterals $G_n$ in anticlockwise order such that $\zeta_n$ and $\zeta_{n+1}$ are two edges of $G_n$, where $0\leq n<5$ and $\zeta_5=\zeta_0$ (see the picture in Figure \ref{Fig_annu-to-disk} on the right).

\begin{figure}[!htpb]
  \setlength{\unitlength}{1mm}
  \centering
  \includegraphics[width=120mm]{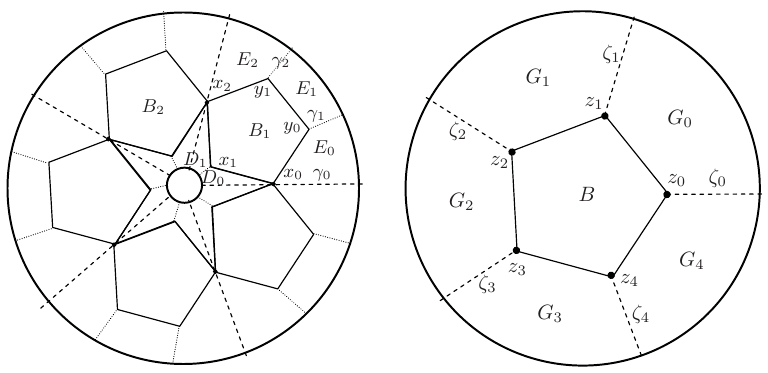}
  \caption{Sketch illustrating of the construction of a branch covering from a round annulus to a round disk with degree 5, where $\ell=2$ and $m=3$.}
  \label{Fig_annu-to-disk}
\end{figure}

We first focus on the extension $F:E_0\rightarrow G_0$. For $x,y\in\T_r$, we use $L(x,y)\subset\T_r$ to denote the component of $\T_r\setminus\{x,y\}$ with smaller arc length. The boundary of $E_0$ consists of the segments $\gamma_0$, $[x_0,y_0]$, $\gamma_1$ and the arc $L(1,e^{2\pi\ii/15})$ while the boundary of $G_0$ consists of the segments $\zeta_0$, $[z_0,z_1]$, $\zeta_1$ and the arc $L(1,e^{2\pi\ii/5})$. It is known that the restriction of $F$ on the arc $L(1,e^{2\pi\ii/15})$ is defined by $F(z)=z^3$. We define $F(\gamma_0, [x_0,y_0], \gamma_1)=(\zeta_0, [z_0,z_1], \zeta_1)$ by three affine maps which keep the corresponding vertices. Then $F$ can be quasiconformally extended to the interior of $E_0$ such that $F(E_0)= G_0$. Similarly, define the quasiconformal extension $F(E_1)= G_1$ and $F(E_2)=G_2$ such that $F(\gamma_1, [y_0,y_1], \gamma_2)=(\zeta_1, [z_1,z_2], \zeta_2)$ and $F(\gamma_2, [y_1,x_2], \gamma_3)=(\zeta_2, [z_2,z_3], \zeta_3)$. By the symmetry of the domain, for $z\in E_s$, where $3\leq s<15$, define $F(z)=e^{6\pi\ii/5}F(e^{-2\pi\ii/5}z)$ inductively.

We can use the similar method to extend $F$ on $A\setminus\mathbb{D}_\varepsilon$ such that $F(D_0)=G_4$ and $F(D_1)=G_3$. For $2\leq t<10$ and $z\in D_t$, define $F(z)=e^{-4\pi\ii/5}F(e^{-2\pi\ii/5}z)$ inductively. For $1\leq j\leq 5$, since $B_j$ and $B$ are regular pentagons and the restriction of $F$ on the boundaries of $B_j$ is affine, there exists a natural conformal extension from $B_i$ onto $B$. Up to now, we obtain a branched covering $F:\mathbb{D}\setminus\mathbb{D}_\varepsilon\rightarrow\mathbb{D}$, where $x_0,x_2,x_4,x_6$ and $x_8$ are branched points. By the construction, $F$ is quasiregular since $F$ is locally quasiconformal almost everywhere. Moreover, $F(e^{2\pi\ii/5}z)=e^{6\pi\ii/5}F(z)$.

For $r_1=r_0^{1-1/\ell}$, $r_2=r_0^{1/m}$ and every $z=re^{\ii \theta}\in\overline{\mathbb{A}}_{r_1,r_2}$, where $r\in[r_1, r_2]$ and $\theta\in[0,2\pi)$, define
\begin{equation*}
\widetilde{F}(z)=\frac{1}{r_0}F([\varepsilon+\frac{1-\varepsilon}{r_2-r_1}(r-r_1)]e^{\ii \theta}).
\end{equation*}
Then $\widetilde{F}:\overline{\mathbb{A}}_{r_1,r_2}\rightarrow \overline{\mathbb{D}}_{r_0}$ is the desired extension stated in the lemma.
\end{proof}

Recall that $F(z)=r_0^\ell/z^\ell$ for $z\in\mathbb{D}_{r_0}$ and $F(z)=z^m$ for $z\in\EC \setminus\overline{\mathbb{D}}$. By the construction in Lemma \ref{an-to-disk}, we obtain the following corollary immediately.

\begin{cor}\label{qr-symm}
The extended map $F:\EC \rightarrow\EC $ is quasiregular with degree $\ell+m$ and satisfies $F(e^{2\pi\ii/(\ell+m)}z)=e^{2m\pi\ii/(\ell+m)}F(z)$ for all $z\in\EC $.
\end{cor}

\begin{prop}\label{straightening}
There exists a quasiconformal mapping $\varphi:\EC \rightarrow\EC $ such that $\varphi\circ F\circ\varphi^{-1}=f_\lambda$, where $f_\lambda$ is the McMullen map whose Julia set is a Cantor set of circles. In particular, Theorem \textup{\ref{qs-classifi} (b)} holds.
\end{prop}

\begin{proof}
Define the standard complex structure $\sigma(z)=\sigma_0$ in $\EC \setminus\overline{\mathbb{D}}$. By the construction, almost all points in $\EC $ are iterated into $\EC \setminus\overline{\mathbb{D}}$ under $F$. If $F^{\circ n}(z)\not\in\EC \setminus\overline{\mathbb{D}}$ for all $n\geq 0$, define $\sigma(z)=\sigma_0$. Otherwise, there exists a minimal $n$ such that $F^{\circ n}(z)\in\EC \setminus\overline{\mathbb{D}}$. Then we define $\sigma(z)=(F^{\circ n})^*(\sigma_0)(z)$, which denotes the pull back of the standard complex structure in $\EC \setminus\overline{\mathbb{D}}$. Since all orbits pass through the annulus $\mathbb{A}_{r_1,r_2}$ at most once, it follows that $\sigma$ is a $F-$invariant almost complex structure on $\EC $.

Let $\varphi : (\EC ,\sigma) \rightarrow (\EC ,\sigma_0)$ be the integrating orientation preserving quasiconformal homeomorphism which is normalized by $\varphi(0)=0$ and $\varphi(z)/z\rightarrow 1$ as $z\rightarrow\infty$. So $\varphi\circ F\circ\varphi^{-1}$ is a rational map with degree $\ell+m$ since $F$ maps $\infty$ to $\infty$ with degree $m$ and $0$ is the only pole with multiplicity $\ell$ by the construction of $F$. This means that $\varphi\circ F\circ\varphi^{-1}$ has the form
\begin{equation*}
\varphi\circ F\circ\varphi^{-1}=(z-a_1)(z-a_2)\cdots(z-a_{\ell+m})/z^\ell,
\end{equation*}
where $a_1,\cdots,a_{\ell+m}$ are the images of the set of zeros of $F$ under $\varphi$.

By the symmetry stated in Corollary \ref{qr-symm}, it follows that $\{a_j\}_{j=1}^{\ell+m}$ lie on a circle centered at the origin uniformly. So $\varphi\circ F\circ\varphi^{-1}=f_{\lambda_0}$, where $f_{\lambda_0}(z)=z^m+\lambda_0/z^\ell$ is the McMullen map and $\lambda_0=(-1)^{\ell+m}a_1\cdots a_{\ell+m}$ is small and lies in the McMullen domain.

By \cite[Theorem 11.14]{Hei01}, an orientation preserving homeomorphism between the Riemann sphere $\EC$ is quasisymmetric if and only if it is quasiconformal. This means that the Julia set $J_{\lambda_0}$ is q.s. equivalent to the standard Cantor circles $A_{\ell,m}$. On the other hand, when $\lambda$ lies in the McMullen domain of $f_\lambda$, then $J_\lambda$ is q.s. equivalent to $J_{\lambda_0}$.
\end{proof}

\subsection{Q.S. uniformization of Sierpi\'{n}ski carpets}

For any subset $X\subset\EC $, the \emph{spherical diameter} of $X$ is defined as $\textup{diam}(X)=\sup_{x,y\in X}|x-y|$, where $|\cdot|$ denotes the \emph{spherical metric}. For any subset $X,Y\subset\EC $, let $\textup{dist}(X,Y):=\inf_{x\in X,y\in Y}|x-y|$ be the distance between $X$ and $Y$.

\begin{defi}
A Jordan curve $\gamma\subset\EC $ is called a \emph{quasicircle} if there exists a constant $k\geq 1$ such that for all distinct points $x,y\in\gamma$,
\begin{equation}\label{bound-turning}
\textup{diam}(I)\leq k\,|x-y|,
\end{equation}
where $I$ is the component of $\gamma\setminus\{x,y\}$ with smaller spherical diameter.
\end{defi}

It is well known that a Jordan curve $\gamma$ is a quasicirlce if it is the image of the unit circle under a quasiconformal homeomorphism defined from $\EC $ to itself. If the dilatation of this quasiconformal homeomorphism is bounded by a constant $K$, then this curve $\gamma$ is called a $K$--\textit{quasicircle}. Moreover, the constants $K$ and $k$ in \eqref{bound-turning} depend only on each other.

\begin{defi}
A family of Jordan curves $\{\gamma_i\}_{i\in\mathbb{N}}$ is said to consist of \emph{uniform quasicircles} if $\gamma_i$ is a $K$--quasicircle for every $i\in\mathbb{N}$, where $K$ is a constant independent of $i$.
A family of Jordan curves $\{\gamma_i\}_{i\in\mathbb{N}}$ is called \emph{uniformly relatively separated} if there exists a constant $s>0$ such that the \textit{relative distance} $\Delta(\gamma_i,\gamma_j)$ between $\gamma_i$ and $\gamma_j$ satisfies
\begin{equation}\label{uni-rel-sep}
\Delta(\gamma_i,\gamma_j):=\frac{\textup{dist}(\gamma_i,\gamma_j)}{\min\{\textup{diam}(\gamma_i),\textup{diam}(\gamma_j)\}}\geq s,
\end{equation}
whenever $i,j\in\mathbb{N}$ and $i\neq j$.
\end{defi}

\begin{thm}[{\cite[Corollary 1.2]{Bon11}}]\label{Bonk}
If $S$ is a Sierpi\'{n}ski carpet in $\EC$ whose peripheral circles are uniformly relatively separated and consists of uniform quasicircles, then there exists a quasisymmetric mapping which maps $S$ onto a round carpet.
\end{thm}

In order to prove Theorem \ref{qs-classifi}(c), it suffices to show that if the free critical points escape and if $J_\lambda$ is a Sierpi\'{n}ski carpet, then the collection of the boundaries of the Fatou components of $f_\lambda$ forms a family of uniform quasicircles and they are uniformly relatively separated.

Let $A$ be an annulus with non-degenerate boundary components. Then there exists a conformal map sending $A$ to a standard annulus $\{z\in\C:0<r<|z|<1\}$, where $r>0$ is uniquely determined by $A$. As an invariant under conformal maps, the \textit{modulus} of $A$ is defined as $\textup{mod}(A)=\frac{1}{2\pi}\log(1/r)$. A set in $\EC$ is called a \textit{Jordan disk} if it is homeomorphic to the unit disk $\D$ and its boundary is a Jordan curve. Note that the Sierpi\'{n}ski carpets that we will study are contained in $\C$. To simplify the notation, we still use $\dist(\cdot,\cdot)$ and $\diam(\cdot)$, respectively, to denote the Euclidean distance and diameter in $\C$.

\begin{lema}\label{estimation}
Let $U\Subset D\neq \C$ be a pair of simply connected domains. Suppose that $\partial U$ is a $K$-quasicircle and $\textup{mod}(D\setminus \overline{U}) \geq m>0$. Then for any univalent map $f:D\rightarrow \mathbb{C}$, the curve $f(\partial U)$ is a $C(K,m)$--quasicircle, where $C(K,m)$ is a constant depending only on $K$ and $m$.
\end{lema}

\begin{proof}
We claim that there exists a constant $C_1(m)\geq 1$ depending only on $m$ such that for any different $x,y,z,w\in\overline{U}$, then
\begin{equation}\label{disto-z}
\frac{1}{C_1(m)}\frac{|x-y|}{|z-w|}\leq\frac{|f(x)-f(y)|}{|f(z)-f(w)|}\leq C_1(m)\frac{|x-y|}{|z-w|}.
\end{equation}
Indeed, let $\varphi_1:\D\to D$ be a conformal map. Then $\varphi_2:=f\circ\varphi_1:\D\to f(D)$ is also conformal.
Since $\textup{mod}(D\setminus \overline{U}) \geq m>0$, there exists a positive constant $r=r(m)<1$ such that $\varphi_1^{-1}(\overline{U})\subset\overline{\D}_r=\{\zeta:|\zeta|\leq r\}$.
Let $x'$ and $y'$, respectively, be the preimages of $x$ and $y$ in $\varphi_1^{-1}(\overline{U})$ under $\varphi_1$.
By Koebe's distortion theorem, we have
\begin{equation*}
\begin{split}
\frac{|f(x)-f(y)|}{|x-y|}=&~\frac{|\varphi_2(x')-\varphi_2(y')|}{|\varphi_1(x')-\varphi_1(y')|}
=\frac{|\varphi_2(x')-\varphi_2(y')|}{|x'-y'|} \cdot \frac{|x'-y'|}{|\varphi_1(x')-\varphi_1(y')|}\\
\leq &~\frac{\max_{\zeta\in\overline{\D}_r}|\varphi_2'(\zeta)|}{\min_{\zeta\in\overline{\D}_r}|\varphi_1'(\zeta)|}\leq \frac{|\varphi_2'(0)|}{|\varphi_1'(0)|}\cdot\left(\frac{1+r}{1-r}\right)^4.
\end{split}
\end{equation*}
Similarly, we have
\begin{equation*}
\frac{|f(z)-f(w)|}{|z-w|}\geq \frac{|\varphi_2'(0)|}{|\varphi_1'(0)|}\cdot\left(\frac{1-r}{1+r}\right)^4.
\end{equation*}
This implies that \eqref{disto-z} holds if we set $C_1(m)=\big(\tfrac{1+r}{1-r}\big)^8$, where $r=r(m)<1$.

\medskip

Since $\partial U$ is a $K$-quasicircle, there exists a constant $C_2(K)>0$ such that for any different points $z_1,z_2\in \partial U$, one has
\begin{equation}\label{quasi-circle}
\frac{\textup{diam}(I)}{|z_1-z_2|}\leq C_2(K),
\end{equation}
where $I$ is one of the components of $\partial U \setminus\{z_1,z_2\}$ with smaller diameter.

Let $x'',y''$ be two different points on $f(\partial U)$ which divide the quasicircle $f(\partial U)$ into two closed subcurves $\beta$ and $\gamma$. Without loss of generality, we assume that $\gamma\subset f(\partial U)$ is the subcurve with smaller diameter. Moreover, let $z'',w''\in \gamma$ such that diam$(\gamma)=|z''-w''|$. Applying \eqref{disto-z} to $f^{-1}$, we have
\begin{equation}\label{disto-1}
\frac{\diam (\gamma)}{|x''-y''|}=\frac{|z''-w''|}{|x''-y''|}\leq C_1(m)\frac{|f^{-1}(z'')-f^{-1}(w'')|}{|f^{-1}(x'')-f^{-1}(y'')|}.
\end{equation}
Note that $f^{-1}(x'')$ and $f^{-1}(y'')$ divide $\partial U$ into two parts $f^{-1}(\beta)$ and $f^{-1}(\gamma)$ .

We divide the argument into two cases. Case I: If $\text{diam}(f^{-1}(\gamma))\leq \text{diam}(f^{-1}(\beta))$, then by \eqref{quasi-circle} and \eqref{disto-1}, we have
\begin{equation}\label{disto-11}
\frac{\diam (\gamma)}{|x''-y''|}\leq C_1(m)\frac{\textup{diam}(f^{-1}(\gamma))}{|f^{-1}(x'')-f^{-1}(y'')|}\leq C_1(m)C_2(K).
\end{equation}
Case II: If \textup{diam}$(f^{-1}(\gamma))> \text{diam}(f^{-1}(\beta))$, let  $z'',w''$ be two points in $\beta$ such that diam$(\beta)=|z''-w''|$. By \eqref{disto-z} and \eqref{quasi-circle}, we have
\begin{equation}\label{disto-2}
\begin{split}
      \frac{\diam (\gamma)}{|x''-y''|}
& \leq \frac{\diam (\beta)}{|x''-y''|} = \frac{|z''-w''|}{|x''-y''|}\leq C_1(m)\frac{|f^{-1}(z'')-f^{-1}(w'')|}{|f^{-1}(x'')-f^{-1}(y'')|}\\
& \leq C_1(m)\frac{\textup{diam}(f^{-1}(\beta))}{|f^{-1}(x'')-f^{-1}(y'')|}\leq C_1(m)C_2(K).
\end{split}
\end{equation}
The lemma follows by combining \eqref{disto-11} and \eqref{disto-2}.
\end{proof}

\begin{lema}[{\cite[Theorem 2.5]{McM94b}}]\label{mod-diam}
Let $A\subset\mathbb{C}$ be an annulus with core curve $\beta$ and with modulus $\textup{mod}(A)>m>0$. Let $U$ be the bounded component of $\mathbb{C}\setminus A$. Then in the Euclidean metric, $$\textup{dist}(U,\beta)>C(m)\,\textup{diam}(\beta),$$
where $C(m)>0$ is a constant depending only on $m$.
\end{lema}

Lemma \ref{mod-diam} shows that if the lower bound of $\textup{mod}(A)$ is controlled then one can control the lower bound of $\text{dist}(U,\mathbb{C}\setminus \overline{A\cup U})/\text{diam}(U)$.

\begin{thm}\label{unif-quasicircle-sep}
Suppose that the free critical points of $f_\lambda$ escape and $J_\lambda$ is a Sierpi\'{n}ski carpet. Let $\{\gamma_i\}_{i\in\mathbb{N}}$ be the collection of the boundaries of all the Fatou components of $f_\lambda$. Then there exists a constant $K\geq 1$ such every $\gamma_i$ is a $K$--quasicircle and $\{\gamma_i\}_{i\in\mathbb{N}}$ are uniformly relatively separated.
\end{thm}
\begin{proof}
By doing a quasiconformal surgery, it can be shown that the boundary of the immediate basin of infinity $\partial B_\lambda$ is quasiconformally homeomorphic to the unit circle, which is the Julia set of $z\mapsto z^m$. This means that $\partial B_\lambda$ is a quasicircle. Hence as preimages of $\partial B_\lambda$, all $\gamma_i$ are quasicircles if $J_\lambda$ is a Sierpi\'{n}ski carpet (Note that all Fatou components are simply connected).

Choose a Jordan curve $\zeta$ in $\mathbb{C}\setminus \overline{B}_\lambda$ which is close enough to $\partial B_\lambda$ such that the annular region $A_0$ bounding by $\zeta$ and $\partial B_\lambda$ contains no free critical orbits. This is possible since all critical points are attracted by $\infty$. Then there exists a preimage $A_1$ of $A_0$ such that $A_1$ is an annulus whose bounded complementary component is $T_\lambda$.

For each $i\geq 1$, we use $\{\gamma_{i,l}\}_{1\leq l\leq k_i}$ to denote the set of the components of $f_\lambda^{-i}(\partial B_\lambda)\setminus f_\lambda^{-i+1}(\partial B_\lambda)$, where $k_i$ is the number of components of $f_\lambda^{-i}(\partial B_\lambda)\setminus f_\lambda^{-i+1}(\partial B_\lambda)$. In particular, $k_1=1$ and $\gamma_{1,1}=\partial T_\lambda$. For convenience, we set $k_0=1$ and denote $\gamma_{0,1}=\partial B_\lambda$. For each $i\geq 1$, every component of $f_\lambda^{-i}(A_0)$ is an annulus since $A_0$ is disjoint from the forward orbits of the critical points of $f_\lambda$. We use $A_{i,l}$ to denote the component of $f_\lambda^{-i}(A_0)$ whose boundary contains $\gamma_{i,l}$.
Let $U_{i,l}$ be the Fatou component of $f_\lambda$ bounded by $\gamma_{i,l}$ and denote $D_{i,l}=\overline{U}_{i,l}\cup A_{i,l}$. Then $U_{i,l}$ and $D_{i,l}$ are Jordan disks.

By Theorem \ref{E-T-T}(c), there exists $k'\geq 2$ such that all the critical points are iterated to $T_\lambda$ under $f_\lambda^{\circ k'}$. This means that the restriction of $f_\lambda$ on $D_{i,l}$ is conformal if $i\geq k'+1$. Therefore, there exists a constant $\widetilde{m}>0$ such that $\textup{mod}(A_{i,l})\geq \widetilde{m}$ for all $i\geq 0$ and $1\leq l\leq k_i$. By Lemma \ref{estimation}, it follows that $\{\gamma_{i,l}: i\geq 0 \text{ and } 1\leq l\leq k_i\}$ are uniform quasicircles.

\medskip

Let $\zeta$ be close to $\partial B_\lambda$ enough such that the disks in
$\{D_{i,l}:0\leq i\leq k' \text{ and } 1\leq l\leq k_i\}$
are pairwise disjoint.
For $i\geq 0$ and $1\leq l\leq k_i$, let $\beta_{i,l}$ be the core curve in $A_{i,l}$, and $V_{i,l}$ the Jordan disk bounded by $\beta_{i,l}$.
Let $U_{i,l_1}$ and $U_{j,l_2}$ be any two different Fatou components, which are preimages of $B_\lambda$, where $i\geq j\geq 0$.
If $V_{i,l_1}\cap V_{j,l_2}=\emptyset$, then by Lemma \ref{mod-diam}, there exists a constant $C(\tfrac{\widetilde{m}}{2})>0$ such that
\begin{equation*}
\frac{\textup{dist}(\gamma_{i,l_1},\gamma_{j,l_2})}{\min\{\textup{diam}(\gamma_{i,l_1}),\textup{diam}(\gamma_{j,l_2})\}}
\geq\min\left\{\frac{\textup{dist}(\gamma_{i,l_1},\beta_{i,l_1})}{\textup{diam}(\beta_{i,l_1})},
\frac{\textup{dist}(\gamma_{j,l_2},\beta_{j,l_2})}{\textup{diam}(\beta_{j,l_2})}\right\}
\geq C(\tfrac{\widetilde{m}}{2}).
\end{equation*}
If $V_{i,l_1}\cap V_{j,l_2}\neq\emptyset$, by the choice of the curve $\zeta$ we have $i\geq j+k'+1$ and hence $D_{i,l_1}\cap U_{j,l_2}=\emptyset$. Indeed, if $z\in D_{i,l_1}\cap U_{j,l_2}$, then $f_\lambda^{\circ j}(z)\in D_{j-i,l_3}\cap B_\lambda=\emptyset$, where $1\leq l_3\leq k_{j-i}$, which is a contradiction.
If $\diam(\gamma_{i,l_1})\leq \diam(\gamma_{j,l_2})$, by Lemma \ref{mod-diam}, we have
\begin{equation*}
\Delta(\gamma_{i,l_1},\gamma_{j,l_2})
=\frac{\textup{dist}(\gamma_{i,l_1},\gamma_{j,l_2})}{\textup{diam}(\gamma_{i,l_1})}
\geq\frac{\textup{dist}(\gamma_{i,l_1},\beta_{i,l_1})}{\textup{diam}(\beta_{i,l_1})}\geq C(\tfrac{\widetilde{m}}{2}).
\end{equation*}
If $\diam(\gamma_{i,l_1})> \diam(\gamma_{j,l_2})$, then
\begin{equation*}
\Delta(\gamma_{i,l_1},\gamma_{j,l_2})
\geq\frac{\textup{dist}(\gamma_{i,l_1},\gamma_{j,l_2})}{\textup{diam}(\gamma_{i,l_1})}
\geq C(\tfrac{\widetilde{m}}{2}).
\end{equation*}
To sum up, this means that $\{\gamma_{i,l}: i\geq 0 \text{ and } 1\leq l\leq k_i\}$ are uniformly relatively separated.
\end{proof}

Part of the result in Theorem \ref{unif-quasicircle-sep} can be obtained also by applying \cite[Lemma 2.1]{QYZ19}, which asserts that the uniform modulus implies uniformly relatively separated. The proof here is slightly different from the one there.

\begin{proof}[{Proof of Theorem \ref{qs-classifi}}]
(a) holds because of David and Semmes's result (see \cite[Proposition 15.11]{DS97}); (b) holds because of Proposition \ref{straightening}; (c) is the corollary of Theorems \ref{Bonk} and \ref{unif-quasicircle-sep}.
\end{proof}

\section{Quasisymmetrically inequivalent Cantor circles}\label{sec-qs-in-cc}

According to Theorem \ref{qs-classifi}(b), if the Julia set $J_\lambda$ of $f_\lambda(z)=z^m+\lambda/z^\ell$ is a Cantor set of circles, then $J_\lambda$ is q.s. equivalent to the standard Cantor circles $A_{\ell,m}$. A natural question is:  whether any two Cantor circles are q.s. \textit{inequivalent} for different pairs $(\ell,m)$? We will give a partial answer to this question in this section and we will see that this question depends on a very basic algebraic problem.

\subsection{Conformal dimension versus quasisymmetric inequivalence}

Let $X$ be a metric space. Recall that the \textit{conformal dimension} $\dim_{C}(X)$ of $X$ is the infimum of the Hausdorff dimensions of all metric spaces which are q.s. equivalent to $X$. Note that the conformal dimension is an invariant of the quasisymmetric class of a metric space. In the rest of this subsection, in order to emphasis the integers $\ell$, $m$ and for convenience, we omit the parameter $\lambda$ in the subscript and use $J_{\ell,m}$ to denote the Julia set of $f_\lambda(z)=z^m+\lambda/z^\ell$, where $\lambda$ is contained in the McMullen domain such that the corresponding Julia set is a Cantor set of circles. The idea of proving $J_{\ell_1,m_1}$ and $J_{\ell_2,m_2}$ are q.s. inequivalent to each other will be based on showing that they have different conformal dimensions.

Recall that $A_{\ell,m}$ is the standard Cantor circles defined in the introduction. By \cite[\S 3]{HP12a}, we have

\begin{lema}\label{conf-dim}
The conformal dimension $\dim_C (A_{\ell,m})=1+\alpha_{\ell,m}$, where $x=\alpha_{\ell,m}$ is the unique positive root of
\begin{equation*}
\ell^{-x}+m^{-x}=1.
\end{equation*}
In particular, if $J_{\ell,m}$ is a set of Cantor circles, then $\dim_C (J_{\ell,m})=1+\alpha_{\ell,m}$. This means that $1+\alpha_{\ell,m}$ is a lower bound of the Hausdorff dimension of $J_{\ell,m}$.
\end{lema}

\begin{lema}\label{no-solution}
Let $\ell_1,m_1,\ell_2,m_2\geq 2$ be $4$ positive integers satisfying $1/\ell_1+1/m_1<1$ and $1/\ell_2+1/m_2<1$. Suppose that $\ell_1\leq m_1$ and $\ell_1\leq \ell_2\leq m_2$. Then the system of equations with variable $x>0$
\begin{equation}\label{equ-no-solu}
\left\{
\begin{array}{ll}
\ell_1^{-x}+ m_1^{-x}=1 & \\
\ell_2^{-x}+ m_2^{-x}=1 &
\end{array}
\right.
\end{equation}
has no solution in any one of the following cases:
\begin{enumerate}
\item $\ell_1=\ell_2$ and $m_1\neq m_2$;
\item $\ell_1<\ell_2$ and $m_1\leq m_2$;
\item $\ell_1<\ell_2\leq m_2<m_1$ and $\ell_1+m_1=\ell_2+m_2$.
\end{enumerate}
\end{lema}

\begin{proof}
The first two cases are trivial. We only prove (c). Let $N=\ell_1+m_1=\ell_2+m_2$ and $0<x<1$. We consider the function
\begin{equation*}
\varphi_x(y)=(N-y)^{-x}+y^{-x},
\end{equation*}
where $0<y<N$. A direct calculation shows that
\begin{equation*}
\varphi_x'(y)=x[(N-y)^{-x-1}-y^{-x-1}].
\end{equation*}
Therefore, $\varphi_x$ is strictly decreasing on $(0,N/2]$ and strictly increasing on $[N/2,N)$. Since $\ell_1<\ell_2\leq m_2<m_1$ and $N=\ell_1+m_1=\ell_2+m_2$, it follows that $0<\ell_1<\ell_2\leq N/2$. So $\ell_2^{-x}+ m_2^{-x} <\ell_1^{-x}+ m_1^{-x}$ for any $0<x<1$. The proof is complete.
\end{proof}

The following result is an immediate corollary of Lemmas \ref{conf-dim} and \ref{no-solution}.

\begin{cor}\label{qs-ineq}
Let $(\ell_1,m_1)$ and $(\ell_2,m_2)$ be two pairs of integers satisfying $\ell_1\leq m_1$ and $\ell_1\leq \ell_2\leq m_2$. If they satisfy also any one of the three cases in Lemma \ref{no-solution}, then $J_{\ell_1,m_1}$ is not q.s.\,equivalent to $J_{\ell_2,m_2}$.
\end{cor}

We conjecture that \eqref{equ-no-solu} has no solution except $(\ell_1,m_1)=(\ell_2,m_2)$. More specifically, suppose that $J_{\ell_1,m_1}$ and $J_{\ell_2,m_2}$ are two sets of Cantor circles. We conjecture that\footnote{Note that this includes a very special case: $(\ell_1,m_1)\neq (\ell_2,m_2)$ but $(\ell_1,m_1)= (m_2,\ell_2)$, i.e. we conjecture that $J_{\ell,m}$ is not q.s. equivalent to $J_{m,\ell}$ if $\ell\neq m$ although they have the same conformal dimension.} $J_{\ell_1,m_1}$ is q.s. equivalent to $J_{\ell_2,m_2}$ if and only if $(\ell_1,m_1)=(\ell_2,m_2)$.

\subsection{More quasisymmetrically inequivalent Cantor circles}\label{subsec-quasis}

In order to find more rational maps whose Julia sets are Cantor circles, the following Theorem \ref{thm-QYY} was proved in \cite{QYY15}.

\begin{thm}[{\cite{QYY15}}]\label{thm-QYY}
For each $p\in\{0,1\}$ and positive integers $d_1,\cdots,d_n$ with $n\geq 2$ satisfying $\sum_{i=1}^{n}(1/d_i)<1$, there exist suitable parameters $a_1,\cdots,a_{n-1}$ such that the Julia set of
\begin{equation}\label{family-QYY}
f_{p,d_1,\cdots,d_n}(z)=z^{(-1)^{n-p} d_1}\prod_{i=1}^{n-1}(z^{d_i+d_{i+1}}-a_i^{d_i+d_{i+1}})^{(-1)^{n-i-p}}
\end{equation}
is a Cantor set of circles. Moreover, any rational map whose Julia set is a Cantor set of circles must be topologically conjugated to $f_{p,d_1,\cdots,d_n}$ for some $p$ and $d_1,\cdots,d_n$ on their corresponding Julia sets with suitable parameters $a_1,\cdots,a_{n-1}$.
\end{thm}

From the topological point of view, all Cantor circles are the same since they are all topologically equivalent (homeomorphic) to the `standard' Cantor circles $C_{3,3}\times \mathbb{T}$. Theorem \ref{thm-QYY} gives a complete topological classification of the Cantor circles as the Julia sets of rational maps under the dynamical behaviors. For the classifications of Cantor circle Julia sets of rational maps in the sense of quasisymmetric equivalence, one can refer to \cite{QYY16} and \cite{QY20}.

Let $J_{p,d_1,\cdots,d_n}$ be the Julia set of $f_{p,d_1,\cdots,d_n}$ for $n\geq 2$. In the following, we always assume that $a_i$ is chosen like in Theorem \ref{thm-QYY} such that $J_{p,d_1,\cdots,d_n}$ is a Cantor set of circles since we are only interested in this case. Meantime, we assume that $\lambda$ is small enough such the Julia set $J_{\ell,m}$ of the McMullen map $f_\lambda$ defined in \eqref{McMullen} is a Cantor set of circles, where $1/\ell+1/m<1$. If $d_i=n+1$ for every $1\leq i\leq n$, we use $f_n$ to denote $f_{p,n+1,\cdots,n+1}$ and let $J_n$ be its corresponding Julia set.

\begin{prop}\label{conf-dim-new}
The conformal dimension $\dim_C(J_{p,d_1,\cdots,d_n})=1+\alpha_{p,d_1,\cdots,d_n}$, where $x=\alpha_{p,d_1,\cdots,d_n}$ is the unique positive root of $\sum_{i=1}^{n}d_i^{-x}=1$.
In particular, if $d_i=n+1$ for every $1\leq i\leq n$, then $\alpha_n:=\alpha_{p,d_1,\cdots,d_n}=\log(n)/\log(n+1)$. If $n\neq n'$, then $\alpha_n\neq\alpha_{n'}$. If $n\geq 3$, then $\alpha_n\neq \alpha_{\ell,m}$ for every $\ell,m\geq 2$ with $1/\ell+1/m<1$.
\end{prop}

\begin{proof}
According to the proof of \cite[Theorem 1.1]{QYY15}, it follows that the combinatorics of $f_{p,d_1,\cdots,d_n}$ is determined by the data $(d_1,\cdots,d_n)\in\mathbb{N}^n$ in the sense of Ha\"{\i}ssinsky and Pilgrim \cite[$\S$2]{HP12b}. By Propositions 1.1 and 2.2 in \cite{HP12b}, the conformal dimension of the Julia set of $f_{p,d_1,\cdots,d_n}$ is $\dim_C(J_{p,d_1,\cdots,d_n})=1+\alpha_{p,d_1,\cdots,d_n}$, where $x=\alpha_{p,d_1,\cdots,d_n}$ is the unique positive root of $\sum_{i=1}^{n}d_i^{-x}=1$. In particular, if $d_i=n+1$ for every $1\leq i\leq n$, then $\alpha_n:=\alpha_{p,d_1,\cdots,d_n}=\log(n)/\log(n+1)$. This means that $n\neq n'$ is equivalent to $\alpha_n\neq\alpha_{n'}$.

For the last statement, we claim that if $n\geq 3$, then $x=\log(n)/\log(n+1)$ is not the solution of $\ell^{-x}+m^{-x}=1$ for any $\ell,m\geq 2$ with $1/\ell+1/m<1$. Without loss of generality, we assume that $2\leq \ell\leq m$. Then $1/\ell^x\geq 1/m^x$, where $x=\log(n)/\log(n+1)$. If $n\geq 3$, then
\begin{equation*}
\frac{1}{\ell^x}+\frac{1}{m^x}\leq \frac{1}{2^{\log 3/\log 4}}+\frac{1}{3^{\log 3/\log 4}}=0.9960381127\cdots<1
\end{equation*}
since $\log(n-1)/\log(n)<\log(n)/\log(n+1)$. The proof is complete.
\end{proof}

Proposition \ref{conf-dim-new} gives a specific example to verify that there exist hyperbolic rational maps whose Julia sets are Cantor circles and whose conformal dimensions are arbitrarily close to 2 (see \cite[Theorem 2]{HP12b}). If we notice that the conformal dimension is an invariant of the quasisymmetric class of a metric space, Proposition \ref{conf-dim-new} has following immediate corollary.

\begin{cor}
If $n\geq 3$, then $J_n$ is not quasisymmetrically equivalent to any $J_{\ell,m}$ for $1/\ell+1/m<1$.
\end{cor}

\section{The non-escaping case}\label{Non-escap-case}

Let $M$ be the Mandelbrot set. The quadratic polynomial $P_c(z)=z^2+c$, $c\in M$ is called \textit{hyperbolic} if and only if it has a bounded attracting periodic orbit. For any given hyperbolic component $H$ of $M$ and its root point $r_H$, $P_{r_H}$ has a unique parabolic periodic orbit with period $m$. It was known that $p(H)=vm$ for some integer $v\geq 1$ (see \cite[Lemma 6.3]{Mil00c}), where $p(H)$ is the period of the hyperbolic component of $H$ defined in the introduction. A hyperbolic component $H$ of $M$ is called \textit{primitive} if $v=1$ and it is called \textit{satellite} if $v\geq 2$. Intuitively, $H$ is primitive if and only if it is surrounded by a cardioid, or there exists no other hyperbolic components attached to the root of $H$.

\subsection{Renormalization theory}

Let us recall some results about the renormalization theory, which can be found in \cite{DH85b,McM94b}. Suppose that $U$, $V$ are two simply connected domain in $\mathbb{C}$ satisfying $\overline{U}\subset V$. The triple $(g,U,V)$ is called a \textit{polynomial-like} map if $g:U\rightarrow V$ is holomorphic and proper. The \textit{filled-in Julia set} of $(g,U,V)$ is defined as $K(g):=\bigcap_{n\geq 0}g^{-n}(U)$ on which all the iterates of $g$ are well-defined. The quadratic polynomial $P_c$ is called \textit{renormlizable} if there exist $n\geq 2$ and $U, V$ containing 0 but not containing the Julia set $J_{P_c}$, such that $(P_c^{\circ n},U,V)$ is  a polynomial-like map. The integer $n$ is called the \textit{period} of renormalization.

Let $K_j:=P_c^{\circ j}(K_0)$. Then $P_c(K_j)=K_{j+1}$, where $j=0,1,\cdots,n-1$ and $K_n=K_0$. Every $K_j$ is called a \textit{small Julia set} after renormalization and they are subsets of the Julia set $J_{P_c}$ of $P_c$. The interior of $K_i$ and $K_j$ are disjoint if $i\neq j$, where $0\leq i,j<n$. The landing point of the external ray with angle zero on $J_{P_c}$ is a repelling or parabolic fixed point, which is called the $\beta$-\textit{fixed point} of $P_c$. It is known the $\beta$-fixed point is disjoint with the small Julia sets \cite[Theorem 7.10]{McM94b}. If $K_i\cap K_j=\emptyset$ for $0\leq i,j<n$, then the renormalization is called \textit{disjoint} type.

Recall that $H_\heartsuit$ is the hyperbolic component of $M$ with period 1. The following result is a well known folk theorem, which can be derived out from \cite[Theorem 2.4, Lemma 2.7 and \S 6]{Mil00c}.

\begin{lema}\label{well-known}
Let $H$ be a hyperbolic component of $M$ but $H\neq H_\heartsuit$.
\begin{enumerate}
\item For any $c\in H$, then $P_c$ is renormlizable with period $p(H)$;
\item The renormalization is of disjoint type if and only if $H$ is primitive;
\item The closure of the bounded Fatou components of $P_c$ are disjoint to each other if and only if $H$ is primitive, where $c\in H\cup \{r_H\}$ and $r_H$ is the root of $H$.
\end{enumerate}
\end{lema}

\subsection{Maximal homeomorphic map}
Recall that $\Lambda$ defined as in (\ref{Lambda}) is the non-escaping locus of $f_\lambda$. It follows from \cite[\S 7]{Ste06} that every hyperbolic component in $\Lambda$ is a hyperbolic component of a copy of the Mandelbrot set $M$ in $\Lambda$. Let $\Phi_\mathcal{M}:M\rightarrow \mathcal{M}$ be the orientation preserving homeomorphism between the Mandelbrot set and its copy in $\Lambda$ such that $\Phi_\mathcal{M}$ maps each of the hyperbolic components of $M$ onto one of that in $\Lambda$. The homeomorphism is required to satisfy the following two conditions: (i) If $c\in H$ such that $P_c$ has an attracting period orbit with multiplier $\zeta\in\mathbb{D}$, then $\Phi(c)\in \mathcal{H}$ is a parameter such the McMullen map $f_{\Phi(c)}$ has also an attracting periodic orbit with multiplier $\zeta$, where $\mathcal{H}=\Phi_\mathcal{M}(H)$ is the \textit{hyperbolic component} of $\mathcal{M}$. The \textit{root} of $\mathcal{H}$ is defined by $r_\mathcal{H}=\Phi_\mathcal{M}(r_H)$. (ii) The homeomorphism $\Phi_\mathcal{M}:M\rightarrow \mathcal{M}$ is maximal. Specifically, if $\Phi_{\mathcal{M}'}:M\rightarrow \mathcal{M'}$ is a homeomorphism satisfying condition (i), then $\mathcal{M}'\subset\mathcal{M}$.
\begin{defi}
The homeomorphic image $\mathcal{M}$ is called a \emph{maximal} copy of the Mandelbrot set in $\Lambda$ and the map $\Phi_{\mathcal{M}}$ is called a \emph{maximal homeomorphic map}.
\end{defi}
For each hyperbolic component $\mathcal{H}$ in $\Lambda$, there exists a unique maximal copy of the Mandelbrot set in $\Lambda$ containing $\mathcal{H}$.
The hyperbolic component $\mathcal{H}$ is called \textit{satellite} or \textit{primitive} if the inverse of $\mathcal{H}$ under the maximal homeomorphic map is satellite or primitive.

For $\lambda\in\mathbb{C}^*$, recall that the ``filled-in" Julia set of $f_\lambda$ is defined as
\begin{equation*}
K_\lambda=\{z\in\mathbb{C}: \{f_\lambda^{\circ k}(z)\}_{k\geq 0} \text{ is bounded}\}.
\end{equation*}
If $\lambda\in\mathbb{C}^*\setminus\Lambda$, then $K_\lambda=J_\lambda$. If $\lambda\in\Lambda$, there maybe exist Fatou components in $K_\lambda$ whose forward orbits are disjoint with the immediate basin of the infinity $B_\lambda$. The following result is a direct corollary of Lemma \ref{well-known}.

\begin{lema}\label{from-well-known}
Let $\mathcal{H}$ be a hyperbolic component of the maximal copy $\mathcal{M}=\Phi_{\mathcal{M}}(M)$ in $\Lambda$ such that $\mathcal{H}\neq\Phi_{\mathcal{M}}(H_\heartsuit)$. For $\lambda\in \mathcal{H}\cup r_\mathcal{H}$, the closure of the components of $\Int (K_\lambda)$ are disjoint to each other if and only if $\mathcal{H}$ is primitive, where $r_\mathcal{H}$ is the root of $\mathcal{H}$.
\end{lema}

\subsection{Sierpi\'{n}ski carpets in the non-escaping case}

Recall that $\lambda_0$ and $\lambda_1$ are defined in \eqref{equ-lambda-0-1}. Another equivalent definition of $\lambda_0$ and $\lambda_1$ is $\lambda_0=\sup\{\lambda:\lambda\in\Lambda_0\cap\mathbb{R}^+\}$ and $\lambda_1=\inf\{\lambda:\lambda\in\Lambda_\infty\cap\mathbb{R}^+\}$, where $\Lambda_0$ and $\Lambda_\infty$ denote the McMullen domain and the Cantor locus respectively.

\begin{thm}\label{QS-positive-restate}
If $\lambda\in [\lambda_0,\lambda_1]$, then $J_\lambda$ is q.s.\,equivalent to a round carpet if and only if one of the following holds:
\begin{enumerate}
\item The interior $\Int (K_\lambda)=\emptyset$ and $\lambda\neq\lambda_0$; or
\item $\lambda\in\mathcal{H}$, where $\mathcal{H}$ is a primitive component of a maximal copy of $M$ in $\Lambda$ such that $\mathcal{H}$ is not the image of $H_\heartsuit$ under the maximal homeomorphic map.
\end{enumerate}
\end{thm}

\begin{proof}
By Theorem \ref{Qiu-Xie-Yin}, we only need to prove the ``if" part and show that if $\lambda=r_\mathcal{H}$ for a primitive hyperbolic component $\mathcal{H}$ in $\Lambda$, then $J_\lambda$ cannot q.s. equivalent to a round carpet. The latter statement is obvious since if $\mathcal{H}$ is a primitive hyperbolic component, then the Fatou set of $f_{r_\mathcal{H}}$ contains a simply connected parabolic component with a cusp, which is not a quasicircle.

By \cite[Lemma 4.1]{QXY12} or \cite[Lemma 4.15]{Xie11}, there exists a maximal homeomorphic copy $\Phi:M\rightarrow\mathcal{M}$ defined from the Mandelbrot set into $\Lambda$, such that $\mathcal{M}\cap \mathbb{R}^+=[\lambda_0,\lambda_1]$. In particular, $\Phi([-2,1/4])=[\lambda_0,\lambda_1]$. Moreover, for each $c\in[-2,1/4]$, there exists a quasiconformal homeomorphism $\Psi_c$ defined in a neighborhood of the Julia set $J_{P_c}$ of $P_c$ such that $\Psi_c(J_{P_c})\subset K_{\Phi(c)}$ and $\Psi_c(J_{P_c})\cap\mathbb{R}=[p_{\Phi(c)},q_{\Phi(c)}]$, where $0<p_\lambda<q_\lambda$ are two real numbers such that $f_\lambda(p_\lambda)=f_\lambda(q_\lambda)=q_\lambda$ whenever $\lambda\in[\lambda_0,\lambda_1]$. In fact, $p_\lambda=\sup\{z:z\in T_\lambda\cap\mathbb{R}^+\}$ and $q_\lambda=\inf\{z:z\in B_\lambda\cap\mathbb{R}^+\}$, where $B_\lambda$ is the immediate attracting basin of $\infty$ and $T_\lambda$ is the Fatou component containing $0$ (see Figure \ref{Fig_McMullen}).

Let $P(f_\lambda):=\textup{Closure}{\{f_\lambda^{\circ k}(\omega_j):0\leq j<\ell+m,k>0\}}$ be the closure of the free postcritical set of $f_\lambda$. From the discussion of last paragraph, we know that $P(f_\lambda)$ is disjoint from $f_\lambda^{-1}(\overline{T}_\lambda)$ since the closure of the forward orbit of the critical point $\omega_0$ is contained in the positive real axis while $f_\lambda^{-1}(\overline{T}_\lambda)\cap\R^+=\emptyset$. By the symmetry of $f_\lambda$, this means that for each component $U$ of $f_\lambda^{-1}(\overline{T}_\lambda)$, there exists a simply connected open neighborhood $V$ of $U$ such that $P(f_\lambda)\cap \overline{V}=\emptyset$.

If $\Int (K_\lambda)=\emptyset$,  except $B_\lambda$, all the Fatou components are iterated onto $T_\lambda$ eventually. By a completely similar argument as in the proof of Theorem \ref{unif-quasicircle-sep}, it can be shown that the peripheral circles of $J_\lambda$ are uniform quasicircles and uniformly relatively separated since the boundary $\partial B_\lambda$ is a quasicircle (see \cite[Lemma 3.8]{QXY12} or \cite[Lemma 4.14]{Xie11}). This means that if $\Int (K_\lambda)=\emptyset$ and $\lambda\neq\lambda_0$, then $J_\lambda$ is q.s. equivalent to a round carpet.

If $\lambda\in\mathcal{H}$ such that $J_\lambda$ is a Sierpi\'{n}ski carpet, then $f_\lambda$ is hyperbolic. A completely similar argument as Theorem \ref{unif-quasicircle-sep} can be shown that  $J_\lambda$ is q.s. equivalent to a round carpet. This completes the proof of Theorems \ref{QS-positive-restate} and \ref{QS-positive}.
\end{proof}

\begin{figure}[!htpb]
  \setlength{\unitlength}{1mm}
  \centering
  \includegraphics[width=55mm]{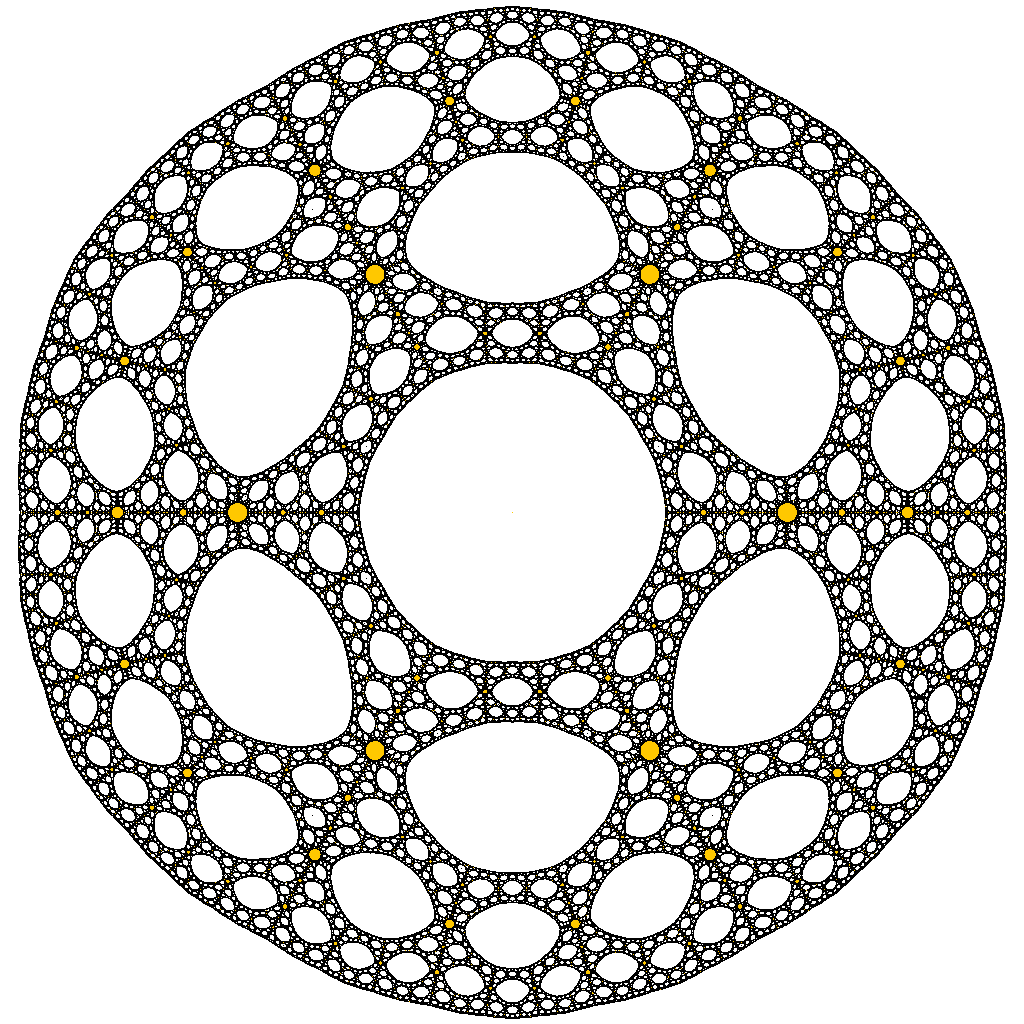}\hskip0.2cm
  \includegraphics[width=55mm]{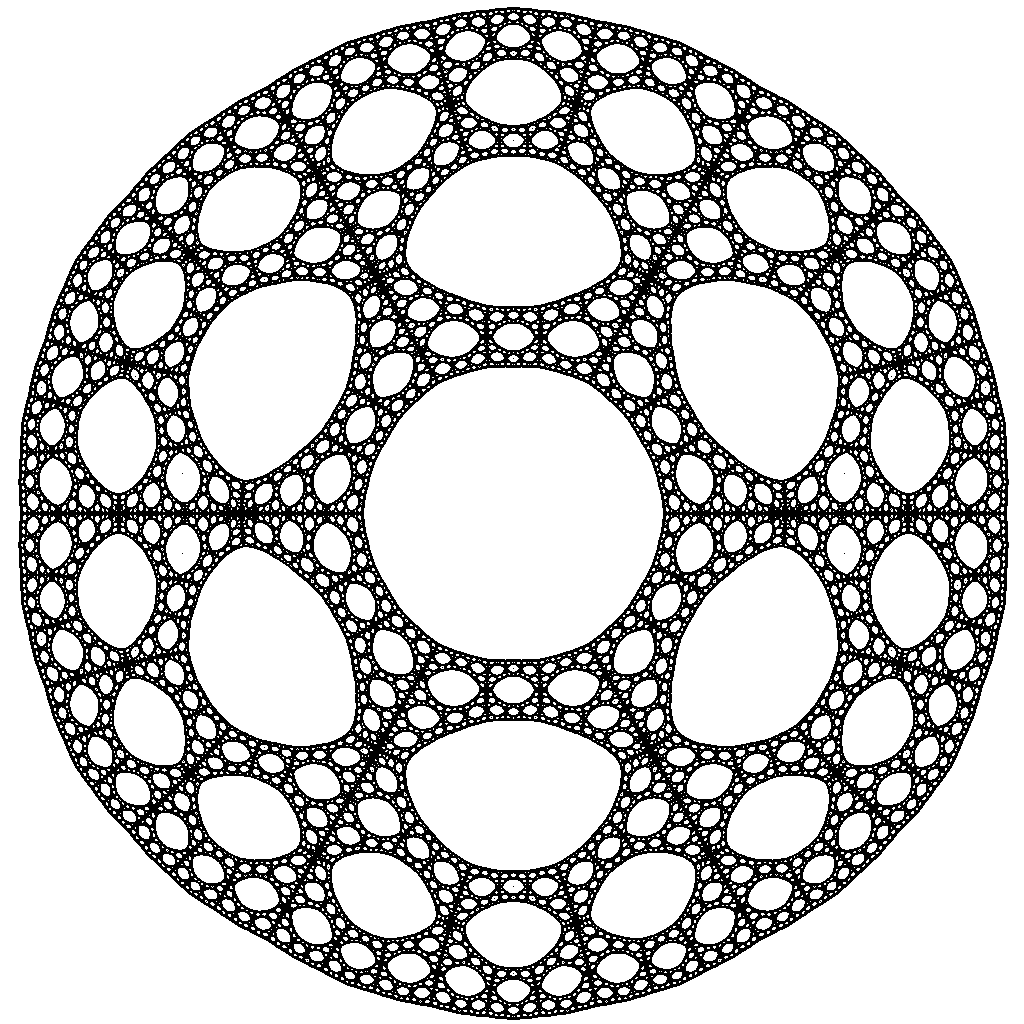}
  \caption{The Julia sets of $f_\lambda(z)=z^3+\lambda/z^3$ with $\lambda_1\approx 0.02749275$ and $\lambda_2\approx 0.02583244$ (from left to right). The parameters $\lambda_1$ and $\lambda_2$ are chosen such that all the free critical points of $f_\lambda$ are iterated to a super-attracting and pre-periodic orbits respectively. Both Julia sets are Sierpi\'{n}ski carpets and they are q.s. equivalent to some round carpets.}
  \label{Fig_McMullen}
\end{figure}

\begin{proof}[{Proof of Corollary \ref{non-hype-carpet}}]
If $\lambda\in(\lambda_0,\lambda_1)$ such that one of the free critical points of $f_\lambda$ is strictly pre-periodic, then all the free critical points of $f_\lambda$ is strictly pre-periodic by the symmetry. Hence $\Int (K_\lambda)=\emptyset$ and $J_\lambda$ is q.s. equivalent to a round carpet by Theorem \ref{QS-positive-restate}(a). Note that the interval $(\lambda_0,\lambda_1)$ in $\Lambda$ corresponds to the interval $(-2,1/4)$ in the Mandelbrot set. Again by Theorem \ref{QS-positive-restate}(a), there exists an infinitely renormalizable $f_\lambda$ whose Julia set is q.s. equivalent to a round carpet.
\end{proof}

In the following, we assume that $n:=\ell=m$ such that the degree of $f_\lambda$ is $2n$, where $n\geq 3$. The following theorem is a weaker version of Theorem 1.3 in \cite{QWY12}.

\begin{thm}[\cite{QWY12}]\label{Qiu-Wang-Yin}
Suppose that the Julia set $J_\lambda$ is connected and the critical orbit of $f_\lambda$ does not accumulate on the boundary $\partial B_\lambda$, then $J_\lambda$ is locally connected.
\end{thm}

Let $\mathcal{P}$ and $\mathcal{Q}$ denote the collections of the Fatou components contained in $\EC \setminus K_\lambda$ and $K_\lambda$ respectively.

\begin{lema}\label{pts-on-bdy-B}
Let $\mathcal{H}$ be a primitive hyperbolic component in $\Lambda$ such that $\mathcal{H}$ is not the image of $H_\heartsuit$ under the maximal homeomorphic map. For $\lambda\in\mathcal{H} \cup \{r_{\mathcal{H}}\}$, then $\overline{U}\cap \overline{B}_\lambda=\emptyset$, where $U\in\mathcal{Q}$ is any Fatou component contained in $K_\lambda$.
\end{lema}

\begin{proof}
Suppose that $\overline{U}\cap \overline{B}_\lambda\neq\emptyset$. Up to iterate several times, we can assume that $U$ is invariant under $f_\lambda$ since $f_\lambda(B_\lambda)=B_\lambda$. Firstly, we claim that $\overline{U}\cap \overline{B}_\lambda=\{z_0\}$ is a fixed point of $f_\lambda$. In fact, since $\partial B_\lambda$ is a Jordan curve \cite[Theorem 1.1]{QWY12}, there exists a homeomorphism $\gamma:\mathbb{T}\rightarrow\partial B_\lambda$ such that $f_\lambda(\gamma(e^{\ii t}))=\gamma(e^{\ii nt})$. If $\sharp (\overline{U}\cap \overline{B}_\lambda)\geq 2$, then there exist $t_1<t_2$ such that $\gamma(e^{\ii t_1}),\gamma(e^{\ii t_2})\in\overline{U}\cap \overline{B}_\lambda$ and $\overline{U}\cap \overline{B}_\lambda\subset\gamma([e^{\ii t_1},e^{\ii t_2}])$, where $0\leq t_2-t_1\leq 2\pi/n$ by the symmetry of the Fatou components of $f_\lambda$. Since $U$ is invariant, then $f_\lambda(\gamma([e^{\ii t_1},e^{\ii t_2}]))=\gamma([e^{\ii t_1},e^{\ii t_2}])$. However, for any subarc $I\subset \overline{B}_\lambda$, there exists $k\geq 0$ such that $f_\lambda^{\circ k}(I)=\partial B_\lambda$ since $\partial B_\lambda$ is a Jordan curve and $f_\lambda$ is conjugated to $z\mapsto z^n$ on $\partial B_\lambda$. This is a contradiction and the claim holds.

From \cite[\S 6]{Ste06}, there exists a maximal copy $\mathcal{M}$ of the Mandelbrot set in the non-escaping locus $\Lambda$, such that $\lambda\in\mathcal{H} \cup \{r_{\mathcal{H}}\}\subset\mathcal{M}$, where $\mathcal{H}$ is a hyperbolic component in $\Lambda$. Let $\Phi:M\rightarrow\mathcal{M}$ be the maximal homeomorphic map as before. Then $\{z_0\}=\overline{U}\cap \overline{B}_\lambda$ is the quasiconformal image of the $\beta$-fixed point of $P_{\Phi^{-1}(\lambda)}$. Since the $\beta$-fixed point is disjoint with the small Julia sets, this contradicts with the assumption $\mathcal{H}\neq \Phi(H_\heartsuit)$.
\end{proof}

\begin{proof}[{Proof of Theorem \ref{QS-complex}}]
We first prove that if $\lambda$ satisfies (a) or (b), then $J_\lambda$ is a Sierpi\'{n}ski carpet. If $\Int (K_\lambda)=\emptyset$, then $\mathcal{Q}_\lambda=\emptyset$. By Theorem \ref{Qiu-Wang-Yin}, we only need to prove that $\overline{V}_1 \cap \overline{V}_2=\emptyset$ for any different $V_1,V_2\in\mathcal{P}$. We first show $\overline{B}_\lambda\cap \overline{T}_\lambda=\emptyset$. If not, one can derive out that $z_0\in\partial B_\lambda\cap \partial T_\lambda$ is a critical point of $f_\lambda$ as the proof of \cite[Propostion 4.3]{DLU05}, which contradicts with the assumption in (a). Suppose that $V_i$ is one of the components of $f_\lambda^{-k_i}(B_\lambda)\setminus f_\lambda^{-(k_i-1)}(B_\lambda)$, where $i=1,2$ and $1\leq k_1\leq k_2$. We divide the argument into two cases. Firstly, if $k_1=k_2$ and $V_1\neq V_2$, then $\overline{V}_1 \cap \overline{V}_2=\emptyset$. Otherwise, there exists some critical point in the forward orbit of $\partial V_1 \cap \partial V_2$, which contradicts with the assumption in (a). If $k_1<k_2$, then $B_\lambda=f_\lambda^{\circ (k_2-1)}(V_1)$ and $T_\lambda=f_\lambda^{\circ (k_2-1)}(V_2)$. This means that $\overline{V}_1 \cap \overline{V}_2=\emptyset$ since $\overline{B}_\lambda \cap \overline{T}_\lambda=\emptyset$. Therefore, $J_\lambda$ is a Sierpi\'{n}ski curve if $\Int (K_\lambda)=\emptyset$ and the closure of the free critical orbits does not intersect with $\partial B_\lambda$.

If $\lambda\in\mathcal{H} \cup \{r_{\mathcal{H}}\}$, where $\mathcal{H}$ is a primitive hyperbolic component of a copy of $M$ in $\Lambda$ such that $\mathcal{H}$ is not the image of $H_\heartsuit$ under the maximal homeomorphic map. By Lemma \ref{from-well-known}, it follows that the closure of the Fatou components in $\mathcal{Q}$ are disjoint to each other. Moreover, a similar argument as in the previous paragraph shows that the closure of the Fatou components in $\mathcal{P}$ are disjoint to each other. We only need prove that the closure of Fatou components in $\mathcal{P}$ and $\mathcal{Q}$ are disjoint to each other.

Let $U\in\mathcal{Q}$ be a Fatou component which lies in $K_\lambda$. By Lemma \ref{pts-on-bdy-B}, it follows that $\overline{U}\cap \overline{V}=\emptyset$ for every $V\in\mathcal{P}$ since $\overline{U}\cap \overline{B}_\lambda=\emptyset$. This means that the closures of any different Fatou components of $f_\lambda$ are disjoint to each other. Therefore, $J_\lambda$ is a Sierpi\'{n}ski carpet.

If $\lambda$ lies in a hyperbolic component of $\Lambda$, then $J_\lambda$ is connected and the critical orbits of $f_\lambda$ cannot intersect with the boundary $\partial B_\lambda$ since the free critical orbits lie in some attracting periodic Fatou components. By applying a similar argument as in Theorem \ref{unif-quasicircle-sep}, it can be shown that the peripheral circles of $J_\lambda$ are uniform quasicircles and they are uniformly relatively separated since the boundary $\partial B_\lambda$ is a quasicircle (see \cite[Theorem 1.2]{QWY12}). If $r_{\mathcal{H}}$ is a root point of a primitive hyperbolic component, then the Fatou set of $f_{r_\mathcal{H}}$ contains a simply connected parabolic component with a cusp, which cannot be a quasicircle. So $J_{r_\mathcal{H}}$ cannot q.s. equivalent to a round carpet.
\end{proof}

\bibliographystyle{amsalpha}
\bibliography{E:/Latex-model/Ref1}

\providecommand{\bysame}{\leavevmode\hbox to3em{\hrulefill}\thinspace}
\providecommand{\MR}{\relax\ifhmode\unskip\space\fi MR }
\providecommand{\MRhref}[2]{%
  \href{http://www.ams.org/mathscinet-getitem?mr=#1}{#2}
}
\providecommand{\href}[2]{#2}
\begin{thebibliography}{QWY12}

\bibitem[BLM16]{BLM16}
M.~Bonk, M.~Lyubich, and S.~Merenkov, \emph{Quasisymmetries of {S}ierpi\'{n}ski
  carpet {J}ulia sets}, Adv. Math. \textbf{301} (2016), 383--422.

\bibitem[Bon11]{Bon11}
M.~Bonk, \emph{Uniformization of {S}ierpi\'{n}ski carpets in the plane},
  Invent. Math. \textbf{186} (2011), no.~3, 559--665.

\bibitem[Dev06]{Dev06}
R.~L. Devaney, \emph{Baby {M}andelbrot sets adorned with halos in families of
  rational maps}, Complex dynamics, Contemp. Math., vol. 396, Amer. Math. Soc.,
  Providence, RI, 2006, pp.~37--50.

\bibitem[Dev13]{Dev13}
\bysame, \emph{Singular perturbations of complex polynomials}, Bull. Amer.
  Math. Soc. (N.S.) \textbf{50} (2013), no.~3, 391--429.

\bibitem[DH85]{DH85b}
A.~Douady and J.~H. Hubbard, \emph{On the dynamics of polynomial-like
  mappings}, Ann. Sci. \'{E}cole Norm. Sup. (4) \textbf{18} (1985), no.~2,
  287--343.

\bibitem[DL05]{DL05}
R.~L. Devaney and D.~M. Look, \emph{Buried {S}ierpinski curve {J}ulia sets},
  Discrete Contin. Dyn. Syst. \textbf{13} (2005), no.~4, 1035--1046.

\bibitem[DLU05]{DLU05}
R.~L. Devaney, D.~M. Look, and D.~Uminsky, \emph{The escape trichotomy for
  singularly perturbed rational maps}, Indiana Univ. Math. J. \textbf{54}
  (2005), no.~6, 1621--1634.

\bibitem[DR13]{DR13}
R.~L. Devaney and E.~D. Russell, \emph{Connectivity of {J}ulia sets for
  singularly perturbed rational maps}, Chaos, CNN, Memristors and Beyond, World
  Scientific, 2013, pp.~239--245.

\bibitem[DS97]{DS97}
G.~David and S.~Semmes, \emph{Fractured fractals and broken dreams}, Oxford
  Lecture Series in Mathematics and its Applications, vol.~7, The Clarendon
  Press, Oxford University Press, New York, 1997.

\bibitem[Hei01]{Hei01}
J.~Heinonen, \emph{Lectures on analysis on metric spaces}, Universitext,
  Springer-Verlag, New York, 2001.

\bibitem[HP12a]{HP12a}
P.~Ha{\"{\i}}ssinsky and K.~M. Pilgrim, \emph{Examples of coarse expanding
  conformal maps}, Discrete Contin. Dyn. Syst. \textbf{32} (2012), no.~7,
  2403--2416.

\bibitem[HP12b]{HP12b}
\bysame, \emph{Quasisymmetrically inequivalent hyperbolic {J}ulia sets}, Rev.
  Mat. Iberoam. \textbf{28} (2012), no.~4, 1025--1034.

\bibitem[McM88]{McM88}
C.~T. McMullen, \emph{Automorphisms of rational maps}, Holomorphic functions
  and moduli, {V}ol. {I} ({B}erkeley, {CA}, 1986), Math. Sci. Res. Inst. Publ.,
  vol.~10, Springer, New York, 1988, pp.~31--60.

\bibitem[McM94]{McM94b}
\bysame, \emph{Complex dynamics and renormalization}, Annals of Mathematics
  Studies, vol. 135, Princeton University Press, Princeton, NJ, 1994.

\bibitem[Mil00]{Mil00c}
J.~Milnor, \emph{Periodic orbits, externals rays and the {M}andelbrot set: an
  expository account}, no. 261, 2000, G\'{e}om\'{e}trie complexe et syst\`emes
  dynamiques (Orsay, 1995), pp.~xiii, 277--333.

\bibitem[QWY12]{QWY12}
W.~Qiu, X.~Wang, and Y.~Yin, \emph{Dynamics of {M}c{M}ullen maps}, Adv. Math.
  \textbf{229} (2012), no.~4, 2525--2577.

\bibitem[QXY12]{QXY12}
W.~Qiu, L.~Xie, and Y.~Yin, \emph{Fatou components and {J}ulia sets of
  singularly perturbed rational maps with positive parameter}, Acta Math. Sin.
  (Engl. Ser.) \textbf{28} (2012), no.~10, 1937--1954.

\bibitem[QY20]{QY20}
W.~Qiu and F.~Yang, \emph{Quasisymmetric uniformization and {H}ausdorff
  dimension of {C}antor circle {J}ulia sets}, arXiv: 1811.10042v2, 2020.

\bibitem[QYY15]{QYY15}
W.~Qiu, F.~Yang, and Y.~Yin, \emph{Rational maps whose {J}ulia sets are
  {C}antor circles}, Ergodic Theory Dynam. Systems \textbf{35} (2015), no.~2,
  499--529.

\bibitem[QYY16]{QYY16}
\bysame, \emph{Quasisymmetric geometry of the {C}antor circles as the {J}ulia
  sets of rational maps}, Discrete Contin. Dyn. Syst. \textbf{36} (2016),
  no.~6, 3375--3416.

\bibitem[QYZ19]{QYZ19}
W.~Qiu, F.~Yang, and J.~Zeng, \emph{Quasisymmetric geometry of {S}ierpi\'{n}ski
  carpet {J}ulia sets}, Fund. Math. \textbf{244} (2019), no.~1, 73--107.

\bibitem[Ste06]{Ste06}
N.~Steinmetz, \emph{On the dynamics of the {M}c{M}ullen family
  {$R(z)=z^m+\lambda/z^{\ell}$}}, Conform. Geom. Dyn. \textbf{10} (2006),
  159--183.

\bibitem[Why58]{Why58}
G.~T. Whyburn, \emph{Topological characterization of the {S}ierpi\'{n}ski
  curve}, Fund. Math. \textbf{45} (1958), 320--324.

\bibitem[Xie11]{Xie11}
L.~Xie, \emph{The dynamics of {M}cmullen maps with real positive parameter},
  Ph.D. Thesis (in Chinese), Fudan University, 2011.

\end{thebibliography}

\end{document}